\documentclass[10pt]{csarticle}
\usepackage{paralist}
\usepackage{hyperref}
\usepackage[small]{caption2}
\usepackage{booktabs}
\usepackage{bbm, boxedminipage}
\usepackage{wasysym}

\newcommand{\ka}{\kappa}

\makeatletter
\def\timenow{\@tempcnta\time
  \@tempcntb\@tempcnta
  \divide\@tempcntb60
  \ifnum10>\@tempcntb0\fi\number\@tempcntb
  \multiply\@tempcntb60
  \advance\@tempcnta-\@tempcntb:\ifnum10>\@tempcnta0\fi\number\@tempcnta}
\makeatother

\newcommand{\Ec}[1]{\ensuremath{\mathbb{E} [#1]}}

{ \end{list} }

\newcommand{\gh}{\mathrm{GH}}

\newcommand{\dgh}{d_{\gh}}
\newcommand{\eqdist}{\ensuremath{\stackrel{d}{=}}}

\renewcommand{\E}[1]{\ensuremath{\mathbb{E} \left[#1 \right]}}
\newcommand{\Prob}[1]{\ensuremath{\mathbb{P} \left(#1 \right)}}
\renewcommand{\I}[1]{\ensuremath{\mathbbm{1}_{ \{ #1 \} }}}

\renewcommand{\subset}{\subseteq}
\newcommand{\convdist}{\ensuremath{\stackrel{d}{\rightarrow}}}

\newcommand{\equidist}{\ensuremath{\stackrel{d}{=}}}

\newcommand{\bl}{\ensuremath{\pmb{\ell}}}
\newcommand{\te}{\ensuremath{\tilde{\mathbf{e}}}}
\newcommand{\be}{\ensuremath{\mathbf{e}}}
\newcommand{\Dir}{\ensuremath{\mathrm{Dirichlet}}}
\newcommand{\Ga}{\ensuremath{\mathrm{Gamma}}}

\hypersetup{
    bookmarks=true,         
    unicode=false,          
    pdftoolbar=true,        
    pdfmenubar=true,        
    pdffitwindow=true,      
    pdftitle={My title},    
    pdfauthor={Author},     
    pdfsubject={Subject},   
    pdfnewwindow=true,      
    pdfkeywords={keywords}, 
    colorlinks=true,       
    linkcolor=blue,          
    citecolor=blue,        
    filecolor=blue,      
    urlcolor=blue           
}

\begin{document}

\title{Critical random graphs: limiting constructions and
  distributional properties\thanks{
\emph{Key words and phrases: random graph; real tree; scaling limit; Gromov--Hausdorff distance; Brownian excursion; continuum random tree; Poisson process; urn model}.  \newline L.A.B.\ was supported by an NSERC Discovery Grant throughout the research and writing of this paper.  C.G.\ was funded by EPSRC Postdoctoral Fellowship EP/D065755/1.}}
\author{L. Addario-Berry \and N. Broutin \and C. Goldschmidt}
\date{March 26, 2010}
\maketitle

\begin{abstract}
We consider the Erd\H{o}s--R\'enyi random graph $G(n,p)$ inside the critical window, where $p=1/n+\lambda n^{-4/3}$ for some $\lambda\in \R$. We proved in \cite{Us1} that considering the connected components of $G(n,p)$ as a sequence of metric spaces with the graph distance rescaled by $n^{-1/3}$ and letting $n \to \infty$ yields a non-trivial sequence of limit metric spaces $\mathcal C=(\mathcal C_1, \mathcal C_2, \dots)$. These limit metric spaces can be constructed from certain random real trees with vertex-identifications.  
For a single such metric space, we give here two equivalent constructions, both of which are in terms of more standard probabilistic objects.  The first is a global construction using Dirichlet random variables and Aldous' Brownian continuum random tree.  The second is a recursive construction from an inhomogeneous Poisson point process on $\R_+$.  
These constructions allow us to characterize the distributions of the masses and lengths in the constituent parts of a limit component when it is decomposed according to its cycle structure. In particular, this strengthens results of \citet{LuPiWi1994}  
by providing precise distributional convergence for the lengths of paths between kernel vertices and the length of a shortest cycle, within any fixed limit component. 
\end{abstract}

\section{Introduction}

The {\em Erd\H{o}s--R\'enyi random graph} $G(n,p)$ is the random graph on vertex set $\{1,2,\dots, n\}$ in which each of the $\binom n 2$ possible edges is present independently of the others with probability $p$. In the 50 years since its introduction \cite{erdos60evolution}, this simple model has given rise to a very rich body of mathematics. (See the books \cite{Bollobas2001,janson00random} for a small sample of this corpus.) In a previous paper \cite{Us1}, we considered the {\em rescaled global structure} of $G(n,p)$ for $p$ in the {\em critical window} -- that is, where $p=1/n+\lambda n^{-4/3}$ for some $\lambda \in \R$ -- when individual components are viewed as metric spaces with the usual graph distance.  (See \cite{Us1} for a discussion of the significance of the random graph phase transition and the critical window.)  The subject of the present paper is the asymptotic behavior of individual components of $G(n,p)$, again viewed as metric spaces, when $p$ is in the critical window. 

Let $\mathcal{C}_1^n, \mathcal C_2^n, \ldots$ be the
connected components of $G(n,p)$ listed in decreasing order of size, with ties 
broken arbitrarily.  Write $\mathcal C^n =
(\mathcal{C}_1^n, \mathcal C_2^n, \ldots)$ and write $n^{-1/3} \mathcal C^n$ to mean the sequence of components viewed as metric spaces with the graph distance in each multiplied by $n^{-1/3}$.  Let $d_{GH}$ be the Gromov--Hausdorff distance between two compact metric spaces (see \cite{Us1} for a definition). 

\begin{thm}[\cite{Us1}] \label{thm:clcrg}
There exists a random sequence $\mathcal C$ of compact metric spaces such that 
as $n \to \infty$,
\[
n^{-1/3} \mathcal C^n \convdist \mathcal C,
\]
where the convergence is in distribution in the distance $d$ specified by
\[
d(\mathcal A, \mathcal B) = \left(\sum_{i=1}^\infty d_{GH}(\mathcal A_i, \mathcal B_i)^4 \right)^{1/4}.
\]
\end{thm} 
We refer to the individual metric spaces in the sequence $\mathcal C$ as the \emph{components} of $\mathcal C$. 
The proof of Theorem~\ref{thm:clcrg} relies on a decomposition of any connected labeled graph $G$ into two parts: a ``canonical'' spanning tree (see \cite{Us1} for a precise definition of this tree), and 
a collection of additional edges which we call \emph{surplus edges}. 
Correspondingly, the limiting sequence of metric spaces has a surprisingly simple description as a collection of random 
real trees (given below) in which certain pairs of vertices have been identified
(vertex-identification being the natural analog of adding a surplus
edge, since edge-lengths are converging to 0 in the limit).

In this paper, we consider the structure of the individual components of the limit $\mathcal C$ in greater detail. 
In the limit, these components have a scaling property which means that, in order to describe the distributional structure of a component, 
only the number of vertex identifications (which we also call the surplus) matters, and not the total mass of the tree in which the identifications take place. 
The major contribution of this paper is the description and justification of two construction procedures for building the components of $\mathcal C$ directly, conditional on their size and surplus.  The importance of these new procedures is that instead of relying on a decomposition of a component into a spanning tree and surplus, they rely on a decomposition according to the cycle structure, which from many points of view is much more natural.  

The procedure we describe first is based on glueing randomly rescaled Brownian CRT's along the edges of a random kernel (see Section \ref{gluedalong} for the definition of a kernel). This procedure is more combinatorial, and implicitly underlying it is a novel finite construction of a component of $G(n,p)$, by first choosing a random kernel and then random doubly-rooted trees which replace the edges of the kernel. (However, we do not spell out the details of the finite construction since it does not lead to any further results.) It is this procedure that yields the strengthening of the results of \citet*{LuPiWi1994}.

The second procedure contains Aldous' stick-breaking inhomogeneous Poisson process construction of the Brownian CRT 
as a special case. Aldous' construction, first described in \cite{aldous91crt1}, has seen numerous extensions and applications, among which the papers of \citet*{aldous94recursive,aldous2005weak,peres04lerw,schweinsberg2009loop} are notable. In particular, in the 
same way that the Brownian CRT arises as the limit of the uniform spanning tree in of $\Z^d$ for $d \geq 4$ (proved in \cite{schweinsberg2009loop}), we expect our generalization to arise as the scaling limit of the components of critical percolation in $\Z^d$ or the $d$-dimensional torus, for large $d$. 

Before we move on to the precise description of the constructions, we introduce them informally and discuss their relationship with various facts about random graphs and the Brownian continuum random tree.

\subsection{Overview of the results} 
A key object in this paper is Aldous' Brownian continuum random tree (CRT) \cite{aldous91crt1,aldous91crt2,aldous93crt3}.  In Section \ref{sec:constructions}, we will give a full definition of the Brownian CRT in the context of \emph{real trees coded by excursions}.  For the moment, however, we will simply note that the Brownian CRT is encoded by a standard Brownian excursion, and give a more readily understood definition using a construction given in \cite{aldous91crt2}.

\medskip
\noindent \textsc{Stick-breaking construction of the Brownian CRT.}  Consider an inhomogeneous Poisson process on $\R^+$ with instantaneous rate $t$ at $t \in \R^+$.  Let $J_1, J_2, \ldots$ be its inter-jump times, in the order they occur ($J_1$ being measured from 0).  Now construct a tree as follows.  First take a (closed) line-segment of length $J_1$.  Then attach another line-segment of length $J_2$ to a uniform position on the first line-segment.  Attach subsequent line-segments at uniform positions on the whole of the structure already created. Finally, take the closure of the object obtained.

\medskip
\noindent The canonical spanning tree appearing in the definition of a component of $\mathcal C$ is not the Brownian CRT, except when the surplus is 0; in general, its distribution is defined instead as a modification (via a change of measure) of the distribution of the Brownian CRT which favors trees encoded by excursions with a large area (see Section~\ref{sec:constructions} for details).  We refer to such a continuum random tree as a \emph{tilted} tree. One of the main points of the present work is to establish strong similarity relationships between tilted trees and the Brownian CRT which go far beyond the change of measure in the definition. 

Our first construction focuses on a \emph{combinatorial} decomposition of a connected graph into its cycle structure (kernel) and the collection of trees obtained by breaking down the component at the vertices of the kernel. In the case of interest here, the trees are randomly rescaled instances of Aldous' Brownian CRT. This is the first direct link between the components of $\mathcal C$ having strictly positive surplus and the Brownian CRT.

As we already mentioned, our second construction extends the stick-breaking construction of the Brownian CRT given above.  We prove that the tilted tree corresponding to a connected component of $\mathcal C$ with a fixed number of surplus edges can be built in a similar way: the difference consists in a bias in the lengths of the first few intervals in the process, but the rate of the Poisson point process used to split the remainder of $\R^+$ remains unchanged. It is important to note that an arbitrary bias in the first lengths does not, in general, give a consistent construction: if the distribution is not exactly right, then the initial, biased lengths will appear too short or too long in comparison to the intervals created by the Poisson process. Indeed, it was not {\em a priori} obvious to the authors that such a distribution must necessarily exist. The consistency of this construction is far from obvious and a fair part of this paper is devoted to proving it. In particular, the results we obtain for the combinatorial construction (kernel/trees) identify the correct distributions for the first lengths. Note in passing that the construction shows that the change of measure in the definition of tilted trees is entirely accounted for by biasing a (random) number of paths in the tree. 

The first few lengths mentioned above are the distances between vertices of the kernel of the component. One can then see the stick-breaking construction as jointly building the trees: each interval chooses a partially-formed tree with probability proportional to the sum of its lengths (the current mass), and then chooses a uniformly random point of attachment in that tree. We show that one can analyze this procedure precisely via a continuous urn process where the bins are the partial trees, each starting initially with one of the first lengths mentioned above. The biased distribution of the initial lengths in the bins ensures that the process builds trees in the combinatorial construction which are not only Brownian CRT's but also have the correct joint distribution of masses. The proof relies on a decoupling argument related to de Finetti's theorem \cite{Finetti1931,Aldous1983}.

\subsection{Plan of the paper} 
The two constructions we have just informally introduced are discussed precisely in Section~\ref{sec:constructions}. Along the way, Section~\ref{sec:constructions} also introduces many of the key concepts and definitions of the paper.  The distributional results are stated in Section~\ref{sec:distributional}.  The remainder of the document is devoted to proofs. In Section~\ref{sec:length_core} we derive the distributions of the lengths in a component of $\mathcal C$ between vertices of the kernel. The stick-breaking construction of a component of $\mathcal C$ is then justified in Section~\ref{subsec:recursive}. Finally, our results about the distributions of masses and lengths of the collection of trees obtained when breaking down the cycle structure at its vertices are proved in Section~\ref{sec:urns_distrib}.

\section{Two constructions}\label{sec:constructions}

Suppose that $G_m^p$ is a (connected) component of $G(n,p)$ conditioned to have size (number of vertices)
$m \le n$.  Theorem~\ref{thm:clcrg} entails that $G_m^p$,
with $m n^{-2/3} \to \sigma$ and $pn \to 1$ as $n \to \infty$ and distances
rescaled by $n^{-1/3}$, converges in distribution to some limiting metric space, 
in the Gromov--Hausdorff sense.  (We shall see that the scaling property mentioned above means 
in particular that it will suffice to consider the case $\sigma=1$.) 
We refer to this limiting metric space as ``a component of $\mathcal C$, conditioned to have total size $\sigma$''. 
From the description as a finite graph limit, it is clear that this distribution should not depend upon where in the sequence 
$\mathcal C$ the component appears, a fact we can also obtain by direct consideration of the limiting object (see below). 

\subsection{The viewpoint of Theorem \ref{thm:clcrg}: vertex identifications within a tilted tree.}\label{vitt}

In this section, we summarize the perspective taken in \cite{Us1} on the structure of a component of $\mathcal C$ conditioned to have total size $\sigma$, as we will need several of the same concepts in this paper. Our presentation in this section owes much to the excellent survey paper of \citet{legall05survey}.  A \emph{real tree} is a compact metric space $(\mathcal T, d)$ such that for all $x, y \in \mathcal T$, 
\begin{itemize}
\item there exists a unique geodesic from $x$ to $y$ i.e.\ there exists a unique isometry $f_{x,y}: [0,d(x,y)] \to \mathcal T$ such that $f_{x,y}(0) = x$ and $f_{x,y}(d(x,y)) = y$.  The image of $f_{x,y}$ is called $\llbracket x,y\rrbracket$;
\item the only non-self-intersecting path from $x$ to $y$ is  $\llbracket x,y\rrbracket$ i.e.\ if $q: [0,1] \to \mathcal T$ is continuous and injective and such that $q(0) = x$ and $q(1) = y$ then $q([0,1]) = \llbracket x,y\rrbracket$.
\end{itemize}
In practice, the picture to have in mind is of a collection of line-segments joined together to make a tree shape, with the caveat that there is nothing in the definition which prevents ``exotic behavior'' such as infinitary branch-points or uncountable numbers of leaves (i.e.\ elements of $\mathcal T$ of degree 1).  Real trees encoded by excursions are the building blocks of the metric spaces with which we will deal in this paper, and we now explain them in detail.  By an \emph{excursion}, we mean a
continuous function $h:[0, \infty) \to \R^+$ such that $h(0) = 0$,
there exists $\sigma < \infty$ such that $h(x) = 0$ for $x > \sigma$
and $h(x) > 0$ for $x \in (0, \sigma)$.  Define a distance $d_h$ on
$[0, \infty)$ via
\[
d_h(x, y) = h(x) + h(y) - 2 \inf_{x \wedge y \le z \le x \vee y} h(z)
\]
and use it to define an equivalence relation: take $x \sim y$ if
$d_h(x,y) = 0$.  Then the quotient space $\mathcal T_h := [0, \sigma]
/ \sim$ endowed with the distance $d_h$ turns out to be a real tree.  We will always think of $\mathcal T_h$ as being rooted at the equivalence class of 0.  When $h$ is random, we call $\mathcal T_h$ a random real tree. The excursion $h$ is often referred to as the \emph{height process} of the tree $\mathcal T_h$. We note that $\mathcal T_h$ comes equipped with a natural \emph{mass measure}, which is the measure induced on $\mathcal T_h$ from Lebesgue measure on $[0,\sigma]$.  By a real tree of \emph{mass} or \emph{size} $\sigma$, we mean a real tree built from an excursion of length $\sigma$.

Aldous' \emph{Brownian
 continuum random tree (CRT)}
\cite{aldous91crt1,aldous91crt2,aldous93crt3} is the real tree
obtained by the above procedure when we take $h = 2 \be$,
where $\be = (\be(x), 0 \le x \le 1)$ is a standard Brownian
excursion.

The limit $\mathcal C = (\mathcal C_1, \mathcal C_2, \ldots)$ is a
sequence of compact metric spaces constructed as follows.  First, take
a standard Brownian motion $(W(t), t \ge 0)$ and use it to define the
processes $(W^{\lambda}(t), t \ge 0)$ and $(B^{\lambda}(t), t \geq 0)$
via
\[
W^{\lambda}(t)  = W(t) + \lambda t - \frac{t^2}{2}, \qquad\mbox{and}\qquad
B^{\lambda}(t)  = W^{\lambda}(t) - \min_{0 \le s \le t} W^{\lambda}(s).
\]
Now take a Poisson point process in $\R^+ \times \R^+$ with intensity
$\frac 1 2 \mathscr{L}_2$, where $\mathscr{L}_2$ is Lebesgue measure in the plane.  The excursions of
$2 B^{\lambda}$ away from zero correspond to the limiting components of $\mathcal C$: each excursion
encodes a random real tree which ``spans'' its component, and the Poisson points which fall under the process (and, in particular, under specific excursions) tell us where to make vertex-identifications in these trees in order to obtain the components themselves. 
We next explain this vertex identification rule in detail.

For a given excursion $h$, let $A_h=\{(x,y): 0 \leq x \leq \sigma, 0 \leq y \leq
h(x)\}$ be the set of points under $h$ and above the $x$-axis.  Let
\[
\ell(\xi)=\ell((x,y)) = \sup\{x' \leq x:y=h(x')\} \qquad
\mbox{and}\qquad r(\xi)=r((x,y)) = \inf\{x' \geq x:y=h(x')\}
\]
be the points of $[0,\sigma]$ nearest to $x$ for which
$h(\ell(\xi))=h(r(\xi))=y$ (see Figure \ref{fig:exctree}).
It is now straightforward to describe how the points of a finite
pointset $\mathcal Q \subset A_h$ can be used to make
vertex-identifications: for $\xi \in \mathcal Q$, we simply identify the equivalence classes
$[x]$ and $[r(x)]$ in $\mathcal T_h$. \label{identification} (It should always be clear that the points we are dealing with in the metric spaces are equivalence classes, and we hereafter drop the square brackets.)  We write $g(h, \mathcal Q)$ for the resulting
``glued'' metric space; the tree metric is altered in the obvious way to accommodate the vertex identifications. 

To obtain the metric spaces in the sequence $\mathcal C$ 
from $2 B^{\lambda}$, we simply make the vertex identifications induced by the 
points of the Poisson point process in $\R^+ \times \R^+$ which fall below $2 B^{\lambda}$, and then 
rearrange the whole sequence in decreasing order of size.  The reader may be somewhat puzzled by the fact that we multiply the process $B^{\lambda}$ by 2 and take a Poisson point process of rate $\frac 1 2$.  It would seem more intuitively natural to use the excursions of $B^{\lambda}$ and a Poisson point process of unit rate.  However, the lengths in the resulting metric spaces would then be too small by a factor of 2.  This is intimately related to the appearance of the factor 2 in the height process of the Brownian CRT.  Further discussion is outside the purview of this paper, but may be found in \cite{Us1}.

\begin{figure}[htb]
\centering
\begin{picture}(420,130)
\put(93,0){\includegraphics[scale=.7]{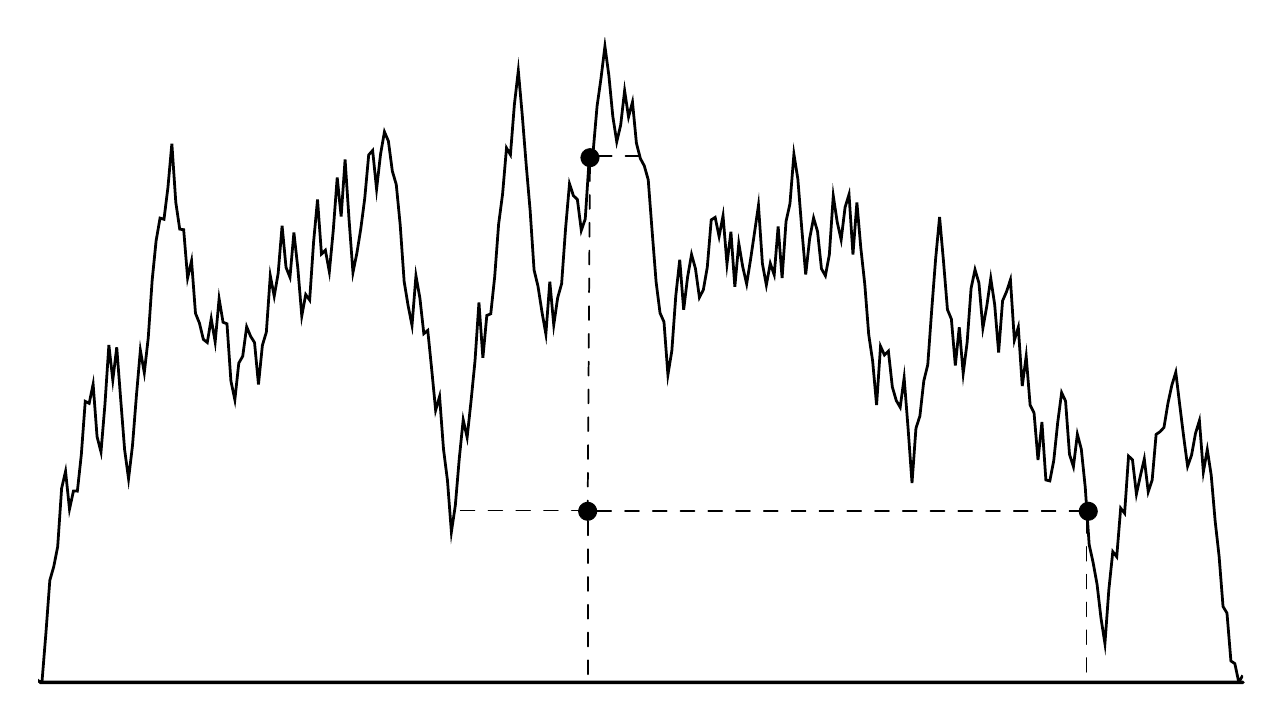}}
\put(208,-2){$x$}
\put(305,-2){$r(x)$}
\put(203,35){$\xi$}
\end{picture}
\caption{A finite excursion $h$ on $[0,1]$ coding a compact real tree $\mathcal T_h$. Horizontal
  lines connect points of the excursion which form equivalence classes in the tree. 
  The point $\xi=(x,y)$ yields the identification of the equivalence classes $[x]$ and $[r(x)]$, which are represented by the horizontal dashed lines.}
\label{fig:exctree}
\end{figure}
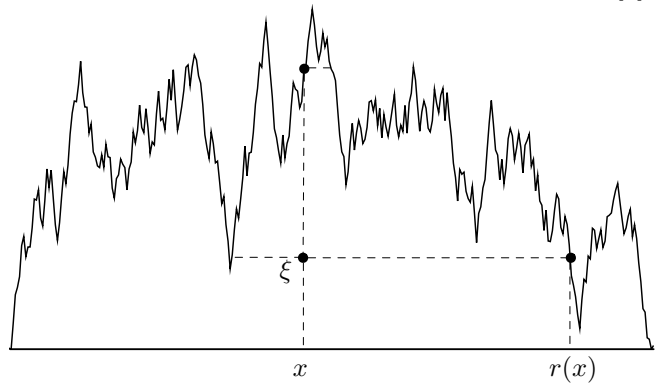

Above, we have described a way of constructing the {\em sequence} of  metric spaces $\mathcal C$. 
In order to see what this implies about a {\em single} component of $\mathcal C$, we must first explain the scaling property of the components $\mathcal C_k$ mentioned above. 
First, consider the excursions above 0 of the process $B^{\lambda}$.  
An excursion theory calculation (see \cite{Us1,aldous97brownian}) shows that, conditional on their lengths, 
the distributions of these excursions do not depend on their starting points.  
Write $\te^{(\sigma)}$ for such an excursion conditioned to have length $\sigma$; in the case $\sigma = 1$,
we will simply write $\te$.  The distribution of $\te^{(\sigma)}$ is most easily described
via a change of measure with respect to the distribution of a Brownian
excursion $\be^{(\sigma)}$ conditioned to have length $\sigma$: for
any test function~$f$,
\begin{equation} \label{eqn:changemeas} \Ec{f(\te^{(\sigma)})} =
  \frac{\E{f(\be^{(\sigma)}) \exp \left( \int_0^{\sigma}
        \be^{(\sigma)}(x) dx \right)} }{\E{ \exp \left(
        \int_0^{\sigma} \be^{(\sigma)}(x) dx \right)}}.
\end{equation}
We refer to $\te^{(\sigma)}$ as a \emph{tilted excursion} and to the tree encoded by $2 \te^{(\sigma)}$ as a \emph{tilted tree}.  The scaling property derives from the fact that a Brownian excursion $\be^{(\sigma)}$ may be obtained from a standard Brownian excursion $\be$ by the transformation $\be^{(\sigma)}(\,\cdot\,) = \sqrt{\sigma} \be (\,\cdot\,/\sigma)$ (Brownian scaling).
Given $\te^{(\sigma)}$, write $\mathcal P$ for the points of a homogeneous Poisson point process of rate $\frac 1 2$ in the plane which fall under the excursion $2 \te^{(\sigma)}$.  Note
that as a consequence of the homogeneity of $\mathcal P$, conditional on $\te^{(\sigma)}$, the number
of points $|\mathcal P|$ has a Poisson distribution with mean
$\int_0^{\sigma} \te^{(\sigma)}(x) dx$.

Let $\be^{(\sigma)}(\,\cdot\,) = \sqrt{\sigma} \be (\,\cdot\,/\sigma)$ as above. 
Then for any test function $f$, by the tower law for conditional expectations we have \label{scalingproperty}
\begin{align*}
\E{f(\te^{(\sigma)})~|~| \mathcal P |=k} 
& = \frac{\E{f(\te^{(\sigma)}) \I{| \mathcal P | = k}}}{\Prob{| \mathcal P |=k}} \\
& = \frac{\E{\E{f(\te^{(\sigma)}) \I{| \mathcal P | = k}~|~\te^{(\sigma)}~}}}{\E{\Prob{| \mathcal P |=k~|~\te^{(\sigma)}}}} \\
& =  \frac{\E{f(\te^{(\sigma)}) \cdot
\frac 1 {k!} \big(\int_0^{\sigma} \te^{(\sigma)}(u) du\big)^k \exp\big(- \int_0^{\sigma} \te^{(\sigma)}(u) du \big) }}
{\E{ \frac1 {k!}\big( \int_0^{\sigma} \te(u) du\big)^k \exp\big(-\int_0^{\sigma} \te^{(\sigma)}(u) du\big)  }} \\
& =  \frac{\E{f(\be^{(\sigma)}) \cdot \big(\int_0^{\sigma} \be^{(\sigma)}(u) du\big)^k}}
{\E{ \big(\int_0^{\sigma} \be^{(\sigma)}(u) du\big)^k }} \\
& = \frac{\E {f(\sqrt{\sigma}\be(\,\cdot\,/\sigma)) \big(\int_0^1 \be(u) du\big)^k} }{\E{ \big(\int_0^1 \be(u) du\big)^k } }.
\end{align*}
Thus, conditional on $|\mathcal P|=k$, the behavior of the tilted excursion of length $\sigma$ may be recovered from that of a tilted excursion of length $1$ by a simple rescaling. 
Throughout the paper, in all calculations that are conditional on the number of Poisson points $|\mathcal P|$, we will take $\sigma=1$ to simplify notation, 
and appeal to the preceding calculation to recover the behavior for other values of $\sigma$.

Finally, suppose that the excursions of $B^{\lambda}$ have ordered
lengths $Z_1 \ge Z_2 \ge \ldots \ge 0$.  Then conditional on their
sizes $Z_1, Z_2, \ldots$ respectively, the metric spaces $\mathcal
C_1, \mathcal C_2, \ldots$ are independent and $\mathcal C_k$ is distributed as $
g(2 \te^{(Z_k)}, \mathcal P)$, $k \ge 1$.  
It follows that one may construct an object distributed as a component of $\mathcal{C}$ conditioned to have size $\sigma$ as follows. 

\medskip
\fbox{%
\mbox{\begin{minipage}[t]{13.5cm}
\textsc{Vertex identifications within a tilted tree}
\begin{enumerate}
\item Sample a tilted excursion $\te^{(\sigma)}$.
\item Sample a set $\mathcal P$ containing 
a $\mathrm{Poisson}\left(\int_0^{\sigma} \te^{(\sigma)}(x) dx\right)$ number of points uniform in the area under $2 \te^{(\sigma)}$.
\item Output $g(2 \te^{(\sigma)}, \mathcal P)$.
\end{enumerate}
\end{minipage}}
}
\medskip

\noindent The validity of this method is immediate from Theorem~\ref{thm:clcrg} and from (\ref{eqn:changemeas}). 

\subsection{Randomly rescaled Brownian CRT's, glued along the edges of a random kernel.}\label{gluedalong}
Before explaining our first construction procedure, we introduce some essential terminology. Our description relies first on a ``top-down'' decomposition of a graph into its cycle structure along with pendant trees, and second on the reverse ``bottom-up'' reconstruction of a component from a properly sampled cycle structure and pendant trees.

\medskip
\noindent\textbf{Graphs and their cycle structure.}
The number of surplus edges, or simply \emph{surplus},
of a connected labeled graph $G=(V,E)$ is defined to be 
$s=s(G)=|E|-|V|+1$. 
In particular, trees have surplus~$0$. We say that the connected graph $G$ is \emph{unicylic} if $s=1$, and \emph{complex} if $s\ge 2$. Define the
{\em core} (sometimes called the \emph{2-core}) $C=C(G)$ to be the
maximum induced subgraph of $G$ which has minimum degree two (so that, in
particular, if $G$ is a tree then $C$ is empty). Clearly the graph
induced by $G$ on the set of vertices $V\setminus V(C)$ is a
forest. So if $u \in V\setminus V(C)$, then there is a unique shortest
path in $G$ from $u$ to some $v \in V(C)$, and we denote this $v$ by
$c(u)$. We extend the function $c(\,\cdot\,)$ to the rest of $V$ by
setting $c(v)=v$ for $v \in V(C)$.

We next define the {\em kernel} $K=K(G)$ to be the multigraph obtained
from $C(G)$ by replacing all paths whose internal vertices all have
degree two in $C$ and whose endpoints have degree at least three in
$C$ by a single edge (see e.g.\ \cite{janson00random}).  If the
surplus $s$ is at most $1$, we agree that the kernel is empty;
otherwise the kernel has minimum degree three and precisely $s-1$ more
edges than vertices.  It follows that the kernel always has at most
$2s$ vertices and at most $3s$ edges.  
We write $\mathrm{mult}(e)$ for the number of copies of an edge $e$ in $K$. 
We now define $\ka(v)$ to be ``the closest bit of $K$ to $v$'', whether that
bit happens to be an edge or a vertex.  Formally, if $v \in V(K)$ we
set $\ka(v)=v$. If $v \in V(C)\setminus V(K)$ then $v$ lies in a path
in $G$ that was contracted to become some copy $e_k$ of an edge $e$ in
$K$; we set $\ka(v) = e_k$. If $v \in V(G) \setminus V(C)$ then we set
$\ka(v)=\ka(c(v))$. In this last case, $\ka(v)$ may be an edge or a
vertex, depending on whether or not $c(v)$ is in $V(K)$. The graphs
induced by $G$ on the sets $\kappa^{-1}(v)$ or $\kappa^{-1}(e_k)$ for
a vertex $v$ or an edge $e_k$ of the kernel $K$ are trees; we call
them \emph{vertex trees} and \emph{edge trees}, respectively, and
denote them $T(v)$ and $T(e_k)$. It will always be clear from
context to which graph they correspond. In each copy $e_k$ of an edge
$uv$, we distinguish in $T(e_k)$ the vertices that are adjacent to $u$
and $v$ on the unique path from $u$ to $v$ in the core $C(G)$, 
and thus view $T(e_k)$ as doubly-rooted. 

Before we define the corresponding notions of core and kernel for the
limit of a connected graph, it is instructive to discuss the
description of a finite connected graph $G$ given in \cite{Us1} (and alluded to 
just after Theorem \ref{thm:clcrg}), and to
see how the core appears in that picture.  Let $G=(V,E)$ be connected
and with ordered vertex set; without loss of generality, we may
suppose that that $V=[m]$ for some $m \geq 1$.  Let $T=T(G)$ be the
so-called \emph{depth-first tree}. This is a spanning tree of the
component which is derived using a certain ``canonical'' version of depth-first search.
(Since the exact nature of that procedure is not important here, we
refer the interested reader to \cite{Us1}.)  Let 
$E^*=E\setminus E(T)$ be the set of surplus edges which must be added
to $T$ in order to obtain $G$.  Let $V^*$ be the set of endpoints of
edges in $E^*$, and let $T_C(G)$ be the union of all shortest paths in
$T(G)$ between elements of $V^*$.  Then the core $C(G)$ is precisely
$T_C(G)$, together with all edges in $E^*$, and $T_C(G)=T(C(G))$.

\medskip
\noindent\textbf{The cycle structure of sparse continuous metric spaces.}
Now consider a real tree $\mathcal T_h$ derived from an excursion $h$,
along with a finite pointset $\mathcal Q \subset A_h$ which specifies
certain vertex-identifications, as described in the previous section.
Let $\mathcal Q_x = \{x:\xi = (x,y) \in \mathcal Q\}$ and let $\mathcal
Q_r = \{r(x) : \xi=(x,y) \in \mathcal Q\}$, both viewed as sets of points
of $\mathcal T_h$. We let $T_C(h,\mathcal Q)$ be the union of all
shortest paths in $\mathcal T_h$ between vertices in the set $\mathcal
Q_x \cup \mathcal Q_r$. Then $T_C(h,\mathcal Q)$ is a subtree of
$\mathcal T_h$, with at most $2|\mathcal Q|$ leaves (this is essentially 
the {\em pre-core} of Section \ref{precore}). We define the
{\em core} $C(h,\mathcal Q)$ of $g(h,\mathcal Q)$ to be the metric
space obtained from $T_C(h,\mathcal Q)$ by identifying $x$ and
$r(x)$ for each $\xi=(x,y) \in \mathcal Q$.  We obtain the {\em
  kernel} $K(h,\mathcal Q)$ from the core $C(h,\mathcal Q)$ by replacing each
maximal path in $C(h,\mathcal Q)$ for which all points but the
endpoints have degree two by an edge. For an edge $uv$ of
$K(h,\mathcal Q)$, we write $\pi(uv)$ for the path in $C(h,\mathcal
Q)$ corresponding to $uv$, and $|\pi(uv)|$ for its length.

For each $x$, let $c(x)$ be the nearest point of
$T_C(h,\mathcal Q)$ to $x$ in $\mathcal T_h$. In other words, $c(x)$
is the point of $T_C(h, \mathcal Q)$ which minimizes $d_h(x,c(x))$.  The nearest bit $\kappa(x)$ of $K(h,\mathcal Q)$ to $x$ 
is then defined in an analogous way to the definition for finite graphs. 
For a vertex $v$ of $K(h,\mathcal Q)$, we define the \emph{vertex tree} $T(v)$ 
to be the subgraph of $g(h,\mathcal Q)$ induced by the points in $\kappa^{-1}(v)=\{x:c(x) = v\}$ 
and the \emph{mass} $\mu(v)$ as the Lebesgue measure of $\kappa^{-1}(v)$.  Similarly, for an edge
$uv$ of the kernel $K(h,\mathcal Q)$ we define the \emph{edge tree}
$T(uv)$ to be the tree induced by $\kappa^{-1}(uv)=\{x:c(x) \in \pi(uv), c(x)\neq u, c(x)\neq v\}\cup
\{u,v\}$ and write $\mu(uv)$ for the Lebesgue measure of $\kappa^{-1}(uv)$.
The two points $u$ and $v$ are considered as distinguished in
$T(uv)$, and so we again view $T(uv)$ as doubly-rooted. 
It is easily seen that these sets are countable unions
of intervals, so their measures are well-defined.  Figures~\ref{fig:reduced_tree} and
\ref{fig:contkernel} illustrate the above definitions.

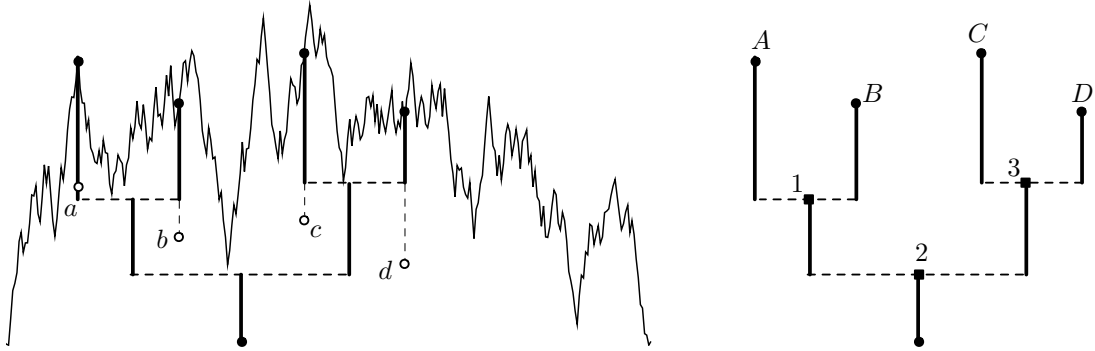
\begin{figure}[htb]
\centering
\begin{picture}(420,130)
\put(0,0){\includegraphics[scale=.7]{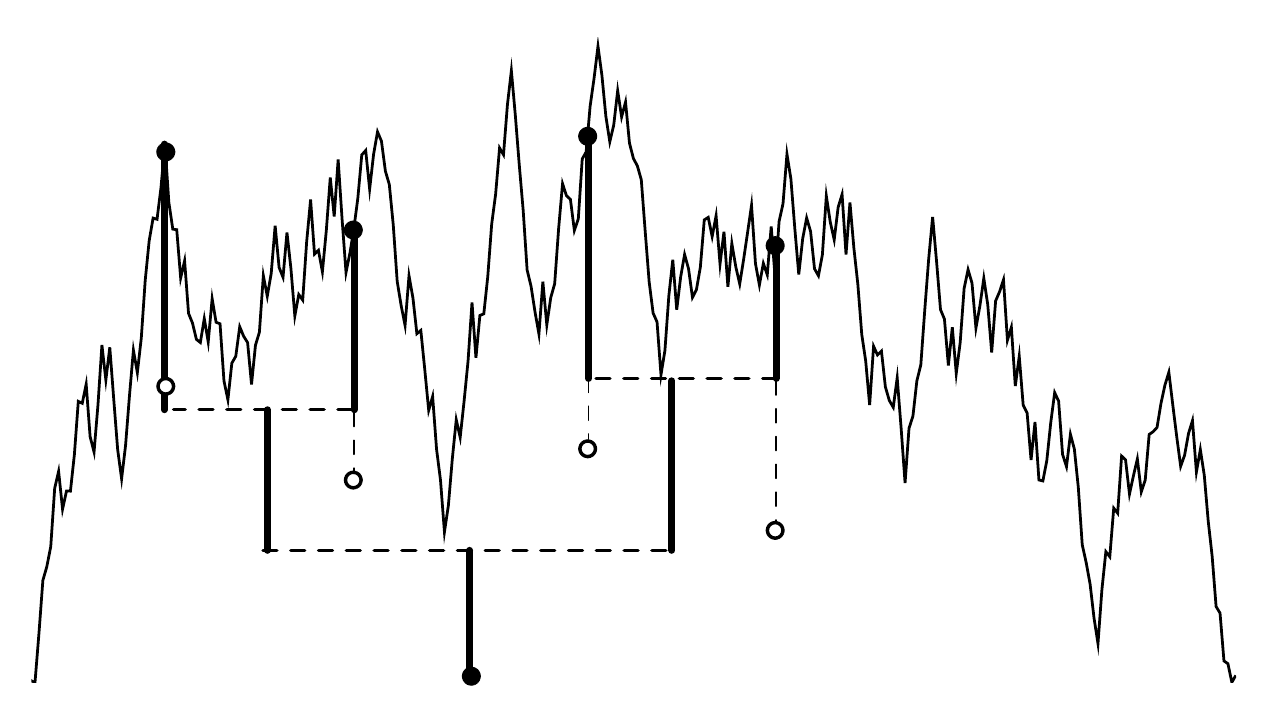}}
\put(280,0){\includegraphics[scale=.7]{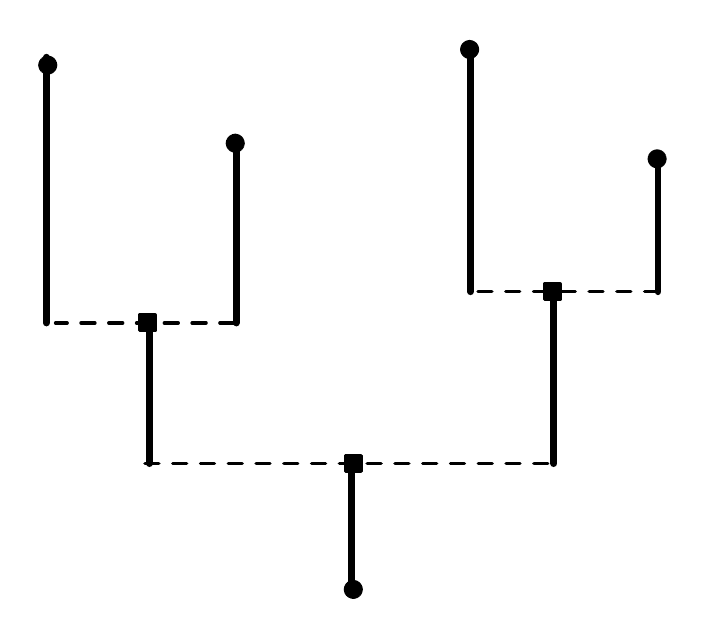}}
\put(28,57){$a$}
\put(63,45){$b$}
\put(121,49){$c$}
\put(147,33){$d$}
\put(288,120){$A$}
\put(330,100){$B$}
\put(370,123){$C$}
\put(409,100){$D$}
\put(303,66){$1$}
\put(350,40){$2$}
\put(385,72){$3$}
\end{picture}
\caption{\label{fig:reduced_tree}An excursion $h$ and the \emph{reduced tree} which is the subtree $T_R(h,\mathcal Q)$ of $\mathcal T_h$ spanned by the root and the leaves $A, B, C, D$ corresponding to the pointset $\mathcal Q=\{a, b, c, d\}$ (which has size $k=4$). The tree $T_R(h, \mathcal Q)$ is a combinatorial tree with edge-lengths.  It will be important in Section~\ref{sec:length_core} below.  It has $2k$ vertices: the root, the leaves and the branch-points $1,2,3$. The dashed lines have zero length. }
\end{figure}

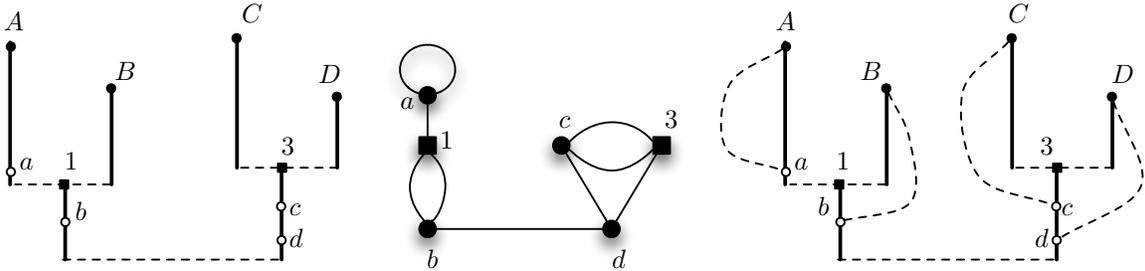
\begin{figure}[htb]
\begin{picture}(420,100)
\put(-7,0){\includegraphics[scale=.7]{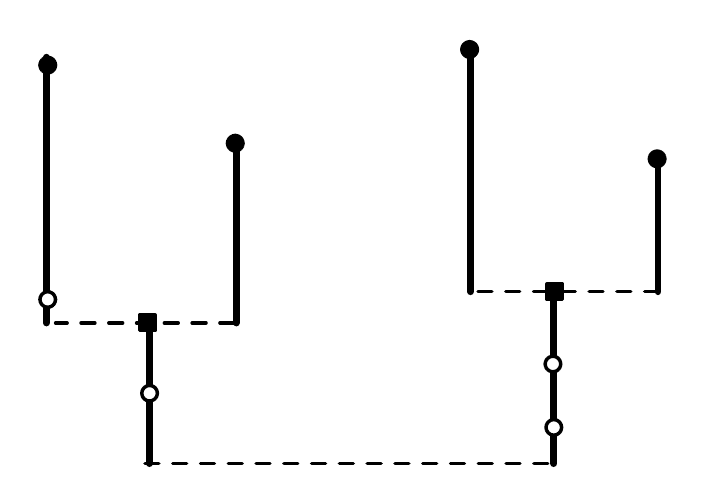}}
\put(0,95){$A$}
\put(42,76){$B$}
\put(90,98){$C$}
\put(119,75){$D$}
\put(6,42){$a$}
\put(27,23){$b$}
\put(108,25){$c$}
\put(108,13){$d$}
\put(23,42){$1$}
\put(105,48){$3$}
\put(263,0){\includegraphics[scale=.7]{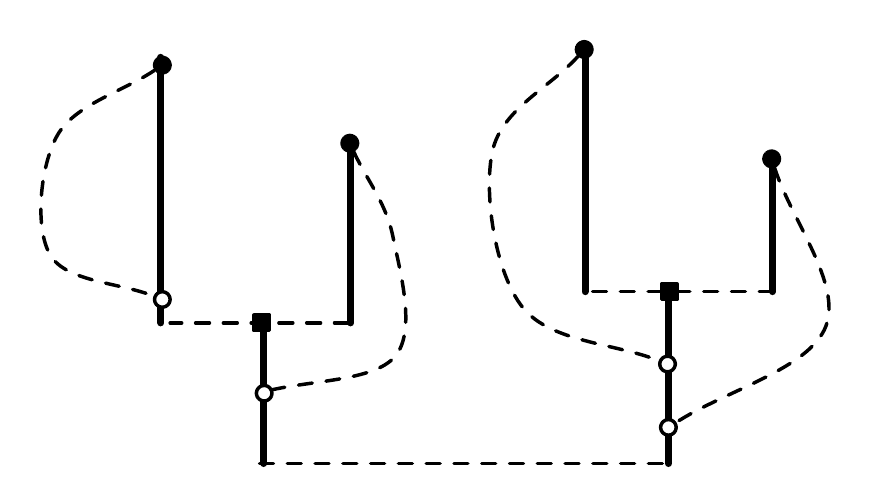}}
\put(292,95){$A$}
\put(324,76){$B$}
\put(380,98){$C$}
\put(419,75){$D$}
\put(299,42){$a$}
\put(308,25){$b$}
\put(400,25){$c$}
\put(390,13){$d$}
\put(315,42){$1$}
\put(392,48){$3$}
\put(135,0){\includegraphics[scale=.7]{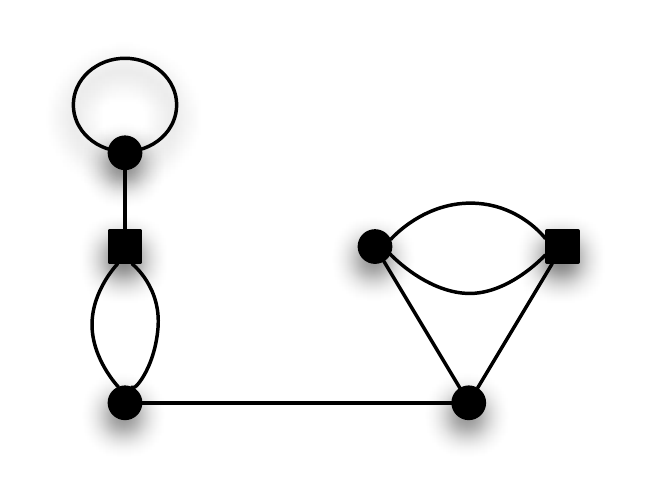}}
\put(150,65){$a$}
\put(165,50){$1$}
\put(160,5){$b$}
\put(230,5){$d$}
\put(250,58){$3$}
\put(210,58){$c$}
\end{picture}
\caption{From left to right: the tree $T_C(h, \mathcal Q)$ from the excursion and pointset of Figure~\ref{fig:reduced_tree}, the corresponding kernel $K(h, \mathcal Q)$ and core $C(h, \mathcal Q)$. The dashed lines indicate vertex identifications.}
\label{fig:contkernel}
\end{figure}

\medskip
\noindent\textbf{Sampling a limit connected component.}
There are two key facts for the first construction procedure. The first is that, for a random metric space $g(2\te, \mathcal P)$ as above, 
conditioned on its mass, an edge tree $T(uv)$ is distributed as a Brownian CRT of mass $\mu(uv)$ and the vertex trees are almost surely empty.  The second is that the kernel $K(2\te, \mathcal P)$ is almost surely 3-regular 
(and so has $2(|\mathcal P|-1)$ vertices and $3(|\mathcal P|-1)$ edges).  Furthermore, for any 3-regular $K$ with $t$ loops,
	\begin{equation} \label{eqn:kerneldist}
	\probC{K(2\te, \mathcal P)=K}{|\mathcal P|} \propto \Bigg(2^t\prod_{e \in E(K)} \mathrm{mult}(e)!\Bigg)^{-1}.
	\end{equation}
(The fact that any reasonable definition of a limit kernel must be 3-regular is obvious from earlier results 
 -- see \cite[][Theorem 7]{janson93birth}, \cite[][Theorem 4]{LuPiWi1994}, and \cite[][Theorems 5.15 and 5.21]{janson00random}. 
Also, (\ref{eqn:kerneldist}) is the limit version of a special case of \cite[][Theorem 7 and (1.1)]{janson93birth}, and is alluded to in \cite{janson00random}, page 128, 
and so is also unsurprising.)
These two facts, together with some additional arguments, will justify the validity of our first procedure for building a component of $\mathcal C$ conditioned to have size $\sigma$, which we now describe. 

Let us condition on $|\mathcal{P}| = k$. As explained before, it then suffices to describe the construction of a component of standard mass $\sigma=1$. 

\medskip
\fbox{%
\begin{minipage}[t]{13.5cm}
\textsc{Procedure 1: randomly rescaled Brownian CRT's}
\begin{itemize}
\item If $k=0$ then let the component simply be a Brownian CRT of total mass 1. 
\item If $k=1$ then let $(X_1,X_2)$ be a $\mathrm{Dirichlet}(\frac 12,\frac 12)$ random vector, let $\mathcal{T}_1,\mathcal{T}_2$ be independent Brownian CRT's of sizes $X_1$ and $X_2$, and identify the root of $\mathcal{T}_1$ with a uniform leaf of $\mathcal T_1$ and with the root of $\mathcal T_2$, to make a ``lollipop'' shape.  
\item If $k \geq 2$ then let $K$ be a random 3-regular graph with $2(k-1)$ vertices chosen according to the probability measure in (\ref{eqn:kerneldist}), above. 
\begin{enumerate}
\item Order the edges of $K$ arbitrarily as $e_1,\ldots,e_{3(k-1)}$, with $e_i=u_iv_i$. 
\item Let $(X_1,\ldots,X_{3(k-1)})$ be a $\mathrm{Dirichlet}(\frac 12,\ldots,\frac 12)$ random vector (see Section \ref{gamdir} for a definition). 
\item Let $\mathcal{T}_1,\ldots,\mathcal{T}_{3(k-1)}$ be independent Brownian CRT's, with tree $\mathcal{T}_i$ having mass $X_i$, and for each $i$ let $r_i$ and $s_i$ be the root and a uniform leaf of $\mathcal T_i$. 
\item Form the component by replacing edge $u_iv_i$ with tree $\mathcal{T}_i$, identifying $r_i$ with $u_i$ and $s_i$ with $v_i$, for $i=1,\ldots,3(k-1)$.
\end{enumerate}
\end{itemize}
\end{minipage}
}
\medskip

\noindent In this description, as in the next, the cases $k=0$ and $k=1$ seem inherently different from the cases $k \geq 2$. In particular, the lollipop shape in the case $k=1$ is a kind of ``rooted core'' that will arise again below. 
For this construction technique, the use of a rooted core seems to be inevitable 
as our methods require us to work with doubly rooted trees. 
Also, as can be seen from the above description, doubly rooted trees are natural objects in the context of a kernel. However, they seem more artificial for graphs whose kernel is empty. 
Finally, we shall see that the use of a rooted core also seems necessary for the second construction technique in the case $k=1$, a fact which is more mysterious to the authors.

\paragraph{An aside: the forest floor picture.}
It is perhaps interesting to pause in order to discuss a rather
different perspective on real trees with vertex identifications.
Suppose first that $\mathcal{T}$ is a Brownian CRT.  Then the path from the root
to a uniformly-chosen leaf has a Rayleigh distribution  \cite{aldous91crt2} (see Section~\ref{gamdir} for a definition of the Rayleigh distribution).  This also the
distribution of the local time at 0 for a standard Brownian bridge.
There is a beautiful correspondence between reflecting Brownian bridge
and Brownian excursion given by \citet{bertoinpitman94path} (also
discussed in \citet{aldouspitman94mappings}), which explains the
connection.

Let $B$ be a standard reflecting Brownian bridge.  Let $L$ be the local time at 0 of $B$, defined by
\[
L_t = \lim_{\epsilon \to 0} \frac{1}{2 \epsilon} \int_0^t \I{B_s < \epsilon} ds.
\] 
Let $U = \sup \{ t \leq 1: L_t = \frac{1}{2} L_1\}$ and let
\[
K_t = \begin{cases} L_t & \text{for $0 \leq t \leq U$} \\
                                    L_1 - L_t & \text{for $U \leq t \leq 1$.}
           \end{cases}
\]

\begin{thm}[\citet{bertoinpitman94path}]
The random variable $U$ is uniformly distributed on $[0,1]$.  Moreover, $X := K + B$
is a standard Brownian excursion, independent of $U$.  Furthermore,
\[
K_t = \begin{cases}
          \min_{t \leq s \leq U} X_s & \text{for $0 \leq t \leq U$} \\
          \min_{U \leq s \leq t} X_s & \text{for $U \leq t \leq 1$.}
          \end{cases}
\]
In particular, $B$ can be recovered from $X$ and $U$.
\end{thm}

So we can think of $\mathcal T$ with its root and uniformly-chosen leaf
as being derived from $X$ and $U$ ($U$ tells us which leaf we
select).  Now imagine the vertices along the path from root to leaf
as a ``forest floor", with little CRT's rooted along its
length.  The theorem tells us that this is properly coded by a
reflecting Brownian bridge.  Distances above the
forest floor in the subtrees are coded by the sub-excursions above 0 of the reflecting bridge; distances along the forest floor are measured in terms of its local time at 0.
This perspective seems natural in the context of the doubly-rooted randomly rescaled CRT's that appear in our second limiting picture. 

There seems to us to be a (so far non-rigorous) connection between this perspective and another technique that has been used for studying random graphs with fixed surplus 
or with a fixed kernel. This technique is to first condition on the core, and then examine the trees that hang off the core. It seems likely that one could directly prove 
that some form of depth- or breadth-first random walk ``along the trees of a core edge'' converges to reflected Brownian motion. In the barely supercritical case 
(i.e.~in $G(n,p)$ when $p=(1+\epsilon(n))/n$ and $n^{1/3}\epsilon(n)\to \infty$ but $\epsilon(n) = o(n^{-1/4})$), \citet*{ding2009} have shown that the ``edge trees'' of the largest component of $G(n,p)$ may 
essentially be generated by the following procedure: start from a path of length given by a geometric with parameter $\epsilon$, then attach to each vertex an independent Galton--Watson tree with Poisson$(1-\epsilon)$ progeny. (We refer the reader to the original paper for a more precise formulation.) The formulation of an analogous result that holds within the critical window seems to us a promising route to 
such a convergence to reflected Brownian motion.

\medskip

We now turn to the second of our constructions for a limiting component conditioned on its size.

\subsection{A stick-breaking construction, run from a random core.}

One of the beguiling features of the Brownian CRT is that it can be constructed in so many different ways.  Here, we will focus on the stick-breaking construction discussed in the introduction.
%
%
%
Aldous~\cite{aldous93crt3} proves that the tree-shape and $2n-1$ branch-lengths created by running this procedure for $n$ steps have the same distribution as the tree-shape and $2n-1$ branch-lengths of the subtree of the Brownian CRT spanned by $n$ uniform points and the root.  This is the notion of  ``random finite-dimensional distributions'' (f.d.d.'s) for continuum random trees.  The sequence of these random f.d.d.'s specifies the distribution of the CRT \cite{aldous93crt3}.
Let $\mathcal A_n$ be the real tree obtained by running the above procedure for $n$ steps (viewed as a metric space). We next prove that $\mathcal A_n$ converges to the Brownian CRT.  This theorem is not new; it simply re-expresses the result of \citet{aldous93crt3} in the Gromov--Hausdorff metric. We include a proof for completeness.
\begin{thm}\label{thm:aldousrecursive}
As $n \to \infty$, $\mathcal A_n$ converges in distribution to the Brownian CRT in the Gromov--Hausdorff distance $\dgh$. 
\end{thm}
\begin{proof}
Label the leaves of the tree $\mathcal A_n$ by the index of their time of addition (so the leaf added at time $J_1$ has label 1, and so on). 
With this leaf-labeling, $\mathcal A_n$ becomes an ordered tree: the first child of an internal node is the one containing the smallest labeled leaf. 
Let $f_n$ be the ordered contour process of $\mathcal A_n$, that is the function (excursion) $f_n:[0,1] \to [0,\infty)$ obtained by recording the distance from the root when 
traveling along the edges of the tree at constant speed, so that each edge is traversed exactly twice, 
the excursion returns to zero at time $1$, and the order of traversal of vertices respects the order of the tree. (See \cite{aldous93crt3} for rigorous details, and \cite{legall05survey} for further explanation.) 
Then by \cite{aldous93crt3}, Theorem 20 and Corollary 22 and Skorohod's representation theorem, 
there exists a probability space on which $\| f_n - 2\be \|_{\infty} \to 0$ almost surely as $n \to \infty$, where $\be$ is a standard Brownian excursion. But $2\be$ is the contour process of the Brownian CRT, 
and by \cite{legall05survey}, Lemma 2.4, convergence of contour processes in the $\|\cdot\|_\infty$ metric implies Gromov--Hausdorff convergence of compact real trees, so $\mathcal A_n$ converges to the Brownian CRT as claimed. 
\end{proof}
We will extend the stick-breaking construction to our random real trees with vertex-identifica-tions. The technical details can be found in Section~\ref{subsec:recursive} but we will summarize our results here.  In the following, let $\mathrm{U}[0,1]$ denote the uniform distribution on $[0,1]$.

\medskip
\fbox{%
\begin{minipage}[t]{13.5cm}
\textsc{Procedure 2: a stick-breaking construction}

First construct a graph with edge-lengths on which to build the component:
\begin{itemize}
\item \textsc{Case $k=0$.} Let $\Gamma=0$ and start the construction from a single point.  
\item \textsc{Case $k = 1$.}  Sample $\Gamma \sim \Ga(\frac 32,\frac 12)$ and $U \sim \mathrm{U}[0,1]$ independently.  Take two line-segments of lengths $\sqrt{\Gamma} U$ and $\sqrt{\Gamma} (1-U)$.  Identify the two ends of the first line-segment and one end of the second.  
\item \textsc{Case $k \ge 2$.}  Let $m = 3k-3$ and sample a kernel $K$ according to the distribution (\ref{eqn:kerneldist}).  Sample $\Gamma \sim \Ga(\frac{m+1}2,\frac 12)$ and $(Y_1,Y_2, \ldots,Y_m) \sim \Dir(1,1,\ldots,1)$ independently of each other and the kernel. Label the edges of $K$ by $\{1,2,\ldots,m\}$ arbitrarily and attach a line-segment of length $\sqrt{\Gamma}Y_i$ in the place of edge $i$, $1 \le i \le m$.  
\end{itemize}
Now run an inhomogeneous Poisson process of rate $t$ at time $t$, conditioned to have its first point at $\sqrt{\Gamma}$.  For each subsequent inter-jump time $J_i$, $i \ge 2$, attach a line-segment of length $J_i$ to a uniformly-chosen point on the object constructed so far.  Finally, take the closure of the object obtained.
\end{minipage}
}
\medskip

\noindent The definitions of the less common distributions used in the procedure appear in Section~\ref{gamdir}.

\begin{thm} \label{thm:recursiveconstruction}
Procedure 2 generates a component with the same distruction as $g(2\te, \mathcal P)$ conditioned to have $|\mathcal P| = k \ge 1$. 
\end{thm}

This theorem implicitly contains information about the total length of the core of $g(2\te, \mathcal P)$: remarkably, conditional upon $|\mathcal P|$, the total length of the core has precisely the right distribution from which to ``start'' the inhomogeneous Poisson process, and our second construction hinges upon this fact.

Our stick-breaking process can also be seen as a continuous urn model, with the $m$ partially-constructed edge trees corresponding to the balls of $m$ different colors in the urn, the probability of adding to a particular edge tree being proportional to the total length of its line segments.  It is convenient to analyze the behavior of this continuous urn model using a discrete one. 
Let $N_1(n), N_2(n), \ldots, N_m(n)$ be the number of balls at step $n$ of P\'olya's urn model started with with one ball of each color, and evolving in such a way that every ball picked is returned to the urn along with two extra balls of the same color \cite{EgPo1923}. Then $N_1(0)=N_2(0)=\dots=N_m(0)=1$, and the vector 
\[\pran{\frac{N_1(n)}{m+2n},\ldots, \frac{N_m(n)}{m+2n}}\] 
converges almost surely to a limit which has distribution $\Dir(\frac 12,\ldots,\frac 12)$ (again, see Section \ref{gamdir} for the definition of this distribution) \cite[][Section~VII.4]{FellerVol2}, \cite[][Chapter~V, Section~9]{AthreyaNey}.
This is also the distribution of the proportions of total mass in each of the edge trees of the component, which is not a coincidence. 
We will see that the stick-breaking process can be viewed as the above P\'olya's urn model performed on the {\em coordinates} of the random vector which keeps track of the proportion of the total mass 
in each of the edge trees as the process evolves.

In closing this section, it is worth noting that the above construction techniques contain a strong dose of both probability and combinatorics. To wit: 
the stick-breaking procedure is probabilistic (but, given the links with urn models, certainly has a combinatorial ``flavor''); the choice of a random kernel conditional 
on its surplus seems entirely combinatorial (but can possibly also be derived from the probabilistic Lemma \ref{lem:treeshapelengths}, below); 
the fact that the edge trees are randomly rescaled CRT's can be derived via either a combinatorial or a 
probabilistic approach (we have taken a probabilistic approach in this paper). 

\section{Distributional results} \label{sec:distributional}
\subsection{Gamma and Dirichlet distributions}\label{gamdir}

Before moving on to state our distributional results, we need to
introduce some relevant notions about Gamma and Dirichlet
distributions.  Suppose that $\alpha, \gamma > 0$.  We say that a random variable has a
$\mathrm{Gamma}(\alpha,\theta)$ distribution if its density function on
$[0,\infty)$ is given by
\[
\frac{\theta^{\alpha} x^{\alpha-1}e^{-\theta x}}{\Gamma(\alpha)}, \quad\text{ where }\quad\Gamma(\alpha)=\int_{0}^{\infty} s^{\alpha-1}e^{-s}ds.
\]
The $\mathrm{Gamma}(1, \theta)$ distribution is the same as the exponential distribution with parameter $\theta$, denoted $\mathrm{Exp}(\theta)$.  Suppose that $a, b > 0$.  We say that a random variable has a $\mathrm{Beta}(a,b)$ distribution if it has density
\[
\frac{\Gamma(a+b)}{\Gamma(a)\Gamma(b)} x^{a-1}(1-x)^{b-1}
\]
on $[0,1]$.  We will make considerable use of the so-called
\emph{beta-gamma algebra} (see \cite{chaumontyor03exercises}, \cite{dufresne98betagamma}), which consists of a collection of
distributional relationships which may be summarized as
\[
\mathrm{Gamma}(\alpha, \theta) \equidist \mathrm{Gamma}(\alpha + \beta, \theta) \times \mathrm{Beta}(\alpha, \beta),
\]
where the terms on the right-hand side are independent.  We will state
various standard lemmas in the course of the text, as we require them.

We write 
\[
\Delta_n = \bigg\{{\bf x} = (x_1, x_2, \ldots, x_n): \sum_{j=1}^n x_j = 1,~x_j > 0, 1 \le j \le n\bigg\}
\]
for the $(n-1)$-dimensional simplex.  For $(\alpha_1, \dots, \alpha_n)
\in \Delta_n$, the $\mathrm{Dirichlet}(\alpha_1, \alpha_2, \ldots, \alpha_n)$ distribution on $\Delta_n$ has density
\[
\frac{\Gamma(\alpha_1 + \alpha_2 + \cdots + \alpha_n)}{\Gamma(\alpha_1) \Gamma(\alpha_2) \ldots \Gamma(\alpha_n)} \cdot \prod_{j=1}^n x_j^{\alpha_n - 1}, 
\]
with respect to $(n-1)$-dimensional Lebesgue measure $\mathscr L_{n-1}$ (so that, in particular, $x_n = 1 - x_1 - x_2 - \cdots - x_{n-1}$).  Fix any $\theta > 0$.  Then if $\Gamma_1, \Gamma_2, \ldots, \Gamma_n$ are independent with $\Gamma_j \sim \mathrm{Gamma}(\alpha_j, \theta)$ and we set
\[
(X_1, X_2, \ldots, X_n) = \frac{1}{\sum_{j=1}^n \Gamma_j} (\Gamma_1, \Gamma_2, \ldots, \Gamma_n),
\]
then $(X_1, X_2, \ldots, X_n) \sim \mathrm{Dirichlet}(\alpha_1, \alpha_2, \ldots, \alpha_n)$, independently of   $\sum_{j=1}^n \Gamma_j\sim \mathrm{Gamma}(\sum_{j=1}^n \alpha_j, \theta)$  (for a proof see, e.g., \cite{JoKoBa1994}).

A random variable has \emph{Rayleigh distribution} if it has density $s
e^{-s^2/2}$ on $[0,\infty)$.  Note that this is the distribution of
the square root of an $\mathrm{Exp}(1/2)$ random variable.  The
significance of the Rayleigh distribution in the present work is that, as mentioned above, it is the distribution of the distance between the root and a
uniformly-chosen point of the Brownian CRT (or, equivalently, between
two uniformly-chosen points of the Brownian CRT) \cite{aldous91crt2}.  We note
here, more generally, that if $\Gamma \sim \Ga(\frac{k+1}2,\frac 12)$ for $k
\ge 0$ then $\sqrt{\Gamma}$ has density
\begin{equation} \label{eqn:sqrtgamma}
\frac{1}{2^{(k-1)/2} \Gamma(\frac{k+1}2)} x^k e^{-x^2/2}.
\end{equation}
Note that in the case $k = 0$, we have $\Gamma \sim \Ga(\frac 12,\frac 12)$ which is the same as the $\chi^2_1$ distribution.  So, as is trivially verified, for $k = 0$, (\ref{eqn:sqrtgamma}) is the density of the modulus of a $\mathrm{Normal}(0,1)$ random variable.

The following relationship between Dirichlet and Rayleigh distributions will be important in the sequel.

\begin{prop} \label{prop:raydir}
Suppose that $(X_1, X_2, \ldots, X_n) \sim \mathrm{Dirichlet}(\frac 12, \ldots, \frac 12)$ independently of $R_1, R_2, \ldots, R_n$ which are i.i.d.\ Rayleigh random variables.  Suppose that $(Y_1, Y_2, \ldots, Y_n) \sim \mathrm{Dirichlet}(1,1,\ldots,1)$, independently of  $\Gamma \sim \mathrm{Gamma}(\frac{n+1}{2}, \frac 12)$.  Then
\[
(R_1 \sqrt{X_1}, R_2 \sqrt{X_2}, \ldots, R_n \sqrt{X_n}) \equidist \sqrt{\Gamma} \times (Y_1, Y_2, \ldots, Y_n).
\]
\end{prop}

\begin{proof}
Firstly, $R_1^2, R_2^2, \ldots, R_n^2$ are independent and identically distributed $\mathrm{Exp}(\frac 12)$ random variables.  Secondly, for any $t > 0$, if $A \sim
\mathrm{Gamma}(t,\frac 12)$ and $B \sim \mathrm{Gamma}(t + \frac 12, \frac 12)$ are independent random variables, then from the gamma duplication formula 
\begin{equation} \label{eqn:gammadup}
AB \equidist C^2,
\end{equation}
where $C \sim \mathrm{Gamma}(2t,1)$ (see, e.g., \cite{wilks1932cga,gordon1994sag}).
So, we can take $R_j = \sqrt{E_j}$, $1 \leq j \leq n$, where
$E_1, E_2, \ldots, E_n$ are independent and identically distributed $\mathrm{Exp}(\frac 12)$ and take
\[
(X_1, X_2, \ldots, X_n) = \frac{1}{\sum_{j=1}^n
  G_j} (G_1, G_2, \ldots, G_n),
\]
 where $G_1, G_2, \ldots, G_n$ are independent and identically distributed
$\mathrm{Gamma}(\frac 12,\frac 12)$ random variables, independent of $E_1, E_2, \ldots,
E_n$.  Note that $\sum_{j=1}^n G_j$ is then also
independent of $(X_1,\ldots,X_n)$ and has $\mathrm{Gamma}(\frac n 2, \frac 12)$
distribution.  It follows that
\[
(R_1 \sqrt{X_1}, R_2 \sqrt{X_2}, \ldots, R_{n}
\sqrt{X_{n}}) = \frac{1}{\sqrt{\sum_{j=1}^{n} G_j}}
(\sqrt{E_1 G_1}, \ldots, \sqrt{E_{n} G_{n}}).
\]
Now by (\ref{eqn:gammadup}), $\sqrt{E_1G_1}, \dots, \sqrt{E_n G_n}$ are independent and distributed as exponential random variables with parameter 1.  So 
\[
(Y_1,\dots, Y_n):=\frac{1}{\sum_{j=1}^{n} \sqrt{E_j G_j}} (\sqrt{E_1 G_1}, \ldots, \sqrt{E_{n} G_{n}}) \sim \mathrm{Dirichlet}(1,1,\ldots, 1),
\]
and $(Y_1,\dots, Y_n)$ is independent of $\sum_{j=1}^{n} \sqrt{E_j G_j}$ which has
$\mathrm{Gamma}(n, 1)$ distribution.  Hence,
\begin{multline*}
\sum_{j=1}^{n} \sqrt{E_j G_j} \times
\frac{1}{\sum_{j=1}^{n} \sqrt{E_j G_j}} (\sqrt{E_1 G_1}, \ldots,
\sqrt{E_{n} G_{n}})\\
= {\textstyle\sqrt{\sum_{j=1}^{n} G_j}} \times \frac{1}{\sqrt{\sum_{j=1}^{n} G_j}}
(\sqrt{E_1 G_1}, \ldots, \sqrt{E_{n} G_{n}}),
\end{multline*}
where the products $\times$ on each side of the equality involve independent random variables.  Applying a Gamma cancellation (Lemma 8 of Pitman \cite{pitman99brownian}), we conclude that
\[
(R_1 \sqrt{X_1}, R_2 \sqrt{X_2}, \dots, R_n \sqrt{X_{n}}) 
\equidist \sqrt{\Gamma} \times (Y_1, Y_2, \ldots, Y_{n}),
\]
where $\Gamma$ is independent of $(Y_1,\dots, Y_n)$ and has a Gamma$(\frac{n+1}{2}, \frac 12)$ distribution.
\end{proof}

\subsection{Distributional properties of the components}

Procedure 1 is essentially a consequence of Theorems~\ref{thm:limit_kernel} and \ref{thm:unicyclic}, below, which capture 
many of the key properties of the metric space $g(2\te, \mathcal P)$ corresponding to the limit of a
connected component of $G(n,p)$ conditioned to have size of order $n^{2/3}$, where $p \sim 1/n$. 
They provide us with a way to sample limit components using only standard objects such as Dirichlet vectors and Brownian CRT's.  

\begin{thm}[\bf Complex components]\label{thm:limit_kernel}
  Conditional on $|\mathcal P|=k\ge 2$, the following statements hold.
\begin{compactenum}[(a)]
	\item The kernel $K(2\te, \mathcal P)$ is almost surely 3-regular (and so has $2(k-1)$ vertices and $3(k-1)$ edges).  For any 3-regular $K$ with $t$ loops,
	\begin{equation}
	\Cprob{K(2\te, \mathcal P)=K}{|\mathcal P|=k} \propto \Bigg(2^t\prod_{e \in E(K)} \mathrm{mult}(e)!\Bigg)^{-1}.
	\end{equation}
	\item For every vertex $v$ of the kernel $K(2\te, \mathcal P)$, we have $\mu(v)=0$ almost surely.
	\item The vector $(\mu(e))$ of masses of the edges $e$ of $K(2\te, \mathcal P)$ has a $\mathrm{Dirichlet}(\frac 12, \dots, \frac 12)$ distribution.
	\item Given the masses $(\mu(e))$ of the edges $e$ of $K(2\te, \mathcal P)$, the metric spaces induced by $g(2\te, \mathcal P)$ on the edge trees are CRT's encoded by independent Brownian excursions of lengths $(\mu(e))$.
	\item For each edge $e$ of the kernel $K(2\te, \mathcal P)$,  the two distinguished points in $\kappa^{-1}(e)$ are independent uniform vertices of the CRT induced by $\kappa^{-1}(e)$.
\end{compactenum} 
\end{thm}

As mentioned earlier, (a) is an easy consequence of \cite[][Theorem 7 and (1.1)]{janson93birth} and \cite[][Theorem 4]{LuPiWi1994} (see also \cite[][Theorems 5.15 and 5.21]{janson00random}).
Also, it should not surprise the reader that the vertex trees are almost surely empty: in the finite-$n$ case, attaching them to the kernel requires only one uniform choice of a vertex (which in the limit becomes a leaf) whereas the edge trees require two such choices.  
The choice of two distinguished vertices has the effect of 
``doubly size-biasing'' the edge trees, making them substantially larger than the singly size-biased vertex trees. 

It turns out that similar results to those in (c)-(e) hold at {\em every} step of the stick-breaking construction and that, roughly speaking, we can view the stick-breaking construction as decoupling into independent stick-breaking 
constructions of rescaled Brownian CRT's along each edge, conditional on the final masses. 
This fact is intimately linked to the ``continuous P\'olya's urn'' perspective mentioned earlier, and also seems related to an extension of the gamma duplication formula due 
to \citet{pitman99brownian}. However, to make all of this precise requires a fair amount of terminology, so we postpone further discussion until later in the paper. 

We note the following corollary about the lengths of the paths in the core of the limit metric space.

\begin{cor} \label{cor:corestuff}
Let $K$ be a 3-regular graph with edge-set $E(K) = \{e_1,e_2,\ldots,e_m\}$ (with arbitrary labeling). 
Then, conditional on $K(2\te, \mathcal P)=K$, the following statements hold.
\begin{compactenum}[(a)]	
	\item Let $(X_1,\dots, X_m)$ be a $\mathrm{Dirichlet}(\frac 12,\dots, \frac 12)$ random vector. Let $R_1, R_2, \ldots, R_m$ be independent and identically distributed Rayleigh random variables. 
Then,
	\[
	(|\pi(e_1)|, |\pi(e_2)|, \ldots, |\pi(e_m)|)\eqdist (R_1\sqrt{X_1}, R_2 \sqrt{X_2}, \ldots, R_m \sqrt{X_m}).
	\]
	\item  Let $\Gamma$ be a $\mathrm{Gamma}\left(\frac{m+1}{2},\frac 12\right)$ random variable. Then, 
	\[
	\sum_{j=1}^m |\pi(e_j)| \eqdist \sqrt{\Gamma} \qquad \text{and}\qquad
	\frac{1}{\sum_{j=1}^m |\pi(e_j)|} (|\pi(e_1)|, |\pi(e_2)|, \ldots, |\pi(e_m)|) \sim \mathrm{Dirichlet}(1,\dots, 1)
	\]
	independently.
\end{compactenum}
\end{cor}
\begin{proof}The distance between two independent uniform leaves of a Brownian CRT is Rayleigh distributed \cite{aldous91crt2}. So the first statement follows from Theorem~\ref{thm:limit_kernel} (iii)--(v) and a Brownian scaling argument.  The second statement follows from Proposition~\ref{prop:raydir}.
\end{proof}

The cases of tree and unicylic components, for which the kernel is
empty, are not handled by Theorem~\ref{thm:limit_kernel}.  The limit
of a tree component is simply the Brownian CRT.  The corresponding
result for a unicyclic component is as follows.

\begin{thm}[\bf Unicyclic components] \label{thm:unicyclic}
  Conditional on $|\mathcal P|=1$, the following statements hold.
\begin{compactenum}[(a)]
\item The length of the unique cycle is
  distributed as the modulus of a $\mathrm{Normal}(0,1)$ random
  variable (by (\ref{eqn:sqrtgamma}) this is also the distribution of
  the square root of a $\Ga(\frac 12,\frac 12)$ random variable).
\item A unicyclic limit component can be generated by sampling $(P_1,P_2) \sim
  \Dir(\frac 12,\frac 12)$, taking two independent Brownian CRT's, rescaling the
  first by $\sqrt{P_1}$ and the second by $\sqrt{P_2}$, 
  identifying the root of the first with a uniformly-chosen vertex in the
  same tree and with the root of the other, to make a lollipop shape. 
\end{compactenum}
\end{thm}

Finally, we note here an intriguing result which is a corollary of Theorem~\ref{thm:limit_kernel} and Theorem 2 of \citet{aldous94recursive}.

\begin{cor}
Take a (rooted) Brownian CRT and sample two uniform leaves.  This gives three subtrees, each of which is marked by a leaf (or the root) and the branch-point.  These doubly-marked subtrees have the same joint distribution as the three doubly-marked subtrees which correspond to the three core edges of $g(2 \te,\mathcal P)$ conditioned to have surplus $|\mathcal P| = 2$.
\end{cor}

In the remainder of this paper, we prove Theorems~\ref{thm:recursiveconstruction}, \ref{thm:limit_kernel} and
\ref{thm:unicyclic} using the limiting picture of \cite{Us1}, described in Section \ref{vitt}. 
Our approach is to start from the core and
then construct the trees which hook into each of the core edges. The lengths in the core are studied in Section~\ref{sec:length_core}. The stick-breaking construction of a limiting component is discussed in Section~\ref{subsec:recursive}. Finally, we use the urn model in order to analyze the stick-breaking construction and to prove the distributional results in Section~\ref{sec:urns_distrib}.

\section{Lengths in the core}\label{sec:length_core}
Suppose we have surplus $|\mathcal P| = k \ge 1$.

If $k \ge 2$ then there are $m = 3(k-1)$ edges in the kernel. Each of these edges corresponds to a path in the core. Let the lengths of these paths be $L_1(0), L_2(0), \ldots, L_m(0)$ (in arbitrary order;
their distribution will turn out to be exchangeable).  Let $C(0) =
\sum_{i=1}^m L_i(0)$ be the total length in the core and let $(P_1(0),
P_2(0), \ldots, P_m(0))$ be the vector of the proportions of this
length in each of the core edges, so that $(L_1(0), L_2(0), \ldots,
L_m(0)) = C(0) \cdot (P_1(0), P_2(0), \ldots, P_m(0))$.  Then we can
rephrase Corollary~\ref{cor:corestuff} as the following collection of
distributional identities:
\begin{align}
C(0)^2  & \sim \mathrm{Gamma}({\textstyle \frac{m+1}2,\frac 12}) \label{eqn:1}\\
(P_1(0), P_2(0), \ldots, P_m(0)) & \sim \Dir(1,1,\ldots,1)  \label{eqn:2} \\
(L_1(0), L_2(0), \ldots, L_m(0)) & \equidist (R_1 \sqrt{P_1}, R_2 \sqrt{P_2}, \ldots, R_m \sqrt{P_m}), \label{eqn:3}
\end{align}
where $C(0)$ is independent of $(P_1(0), P_2(0), \ldots, P_m(0))$ and where $R_1, R_2, \ldots, R_m$ are i.i.d.\ with Rayleigh distribution, independently of $(P_1,P_2, \ldots, P_m) \sim \Dir(\frac 12,\frac 12,\ldots,\frac 12)$.  Of course, (\ref{eqn:3}) follows from (\ref{eqn:1}) and (\ref{eqn:2}) using Proposition~\ref{prop:raydir}.  Although we stated these identities as a corollary of Theorem~\ref{thm:limit_kernel}, proving them will, in fact, be a good starting point for our proofs of Theorem~\ref{thm:limit_kernel}.  

In the case $k = 1$, the core consists of a single cycle.  Write $C(0)$ for its length.  Then we can rephrase part (a) of Theorem~\ref{thm:unicyclic} as
\begin{equation} \label{eqn:4}
C(0)^2 \sim \textstyle \Ga(\frac 12,\frac 12),
\end{equation}
so that, in particular, $C(0)$ is distributed as the modulus of a standard normal random variable.

For the remainder of the section, we will use the notation $T_R(h, \mathcal Q)$, where $h$ is an excursion and $\mathcal Q \subset A_h$ is a finite set of points lying under $h$, 
to mean the so-called \emph{reduced tree} \label{redtree} of $\mathcal{T}_h$, that is the subtree spanned by the root and by the leaves corresponding to the points $\{x:\xi = (x,y) \in \mathcal Q\}$ of the excursion (as defined on page \pageref{identification}).  See Figure~\ref{fig:reduced_tree}.

We first spend some time deriving the effect of the change of measure (\ref{eqn:changemeas}) conditioned on $|\mathcal P|=k$, 
in a manner similar to the explanation of the scaling property on page \pageref{scalingproperty}. We remark that conditional on $\te$ and on $|\mathcal P|=k$, we may view the points of $\mathcal P$ as selected 
by the following two-stage procedure (see Proposition 19 of \cite{Us1})\label{twostage}.  First, choose $\mathbf{V}=(V_1,\ldots,V_k)$, where $V_1,\ldots,V_k$ are independent and identically distributed random variables on $[0,1]$ with density proportional to $\te(u)$. 
Then, given $\mathbf{V}$, choose $\mathbf{W}=(W_1,\ldots,W_k)$ where $W_i$ is uniform on $[0,2\te(V_i)]$, and take $\mathcal P$ to be the set of points $(\mathbf{V},\mathbf{W}) = \{(V_1,W_1),\ldots,(V_k,W_k)\}$. 
Now suppose that $f$ is a non-negative measurable function. Then by the tower law for conditional expectations, we have 
\begin{align*}
\E{f(\te, \mathcal P)~|~| \mathcal P |=k} 
& = \frac{\E{f(\te, \mathcal P) \I{| \mathcal P | = k}}}{\Prob{| \mathcal P |=k}} \\
& = \frac{\E{\E{f(\te,(\mathbf{V},\mathbf{W})) \I{| \mathcal P | = k}~|~\te~}}}{\E{\Prob{| \mathcal P |=k~|~\te}}} \\
& =  \frac{\E{\E{f(\te,(\mathbf{V},\mathbf{W}))~|~\te~} \cdot
\frac 1 {k!} \big(\int_0^1 \te(u) du\big)^k \exp\big(- \int_0^1 \te(u) du \big) }}
{\E{ \frac1 {k!}\big( \int_0^1 \te(u) du\big)^k \exp\big(-\int_0^1 \te(u) du\big)  }}.
\end{align*}
Expressing $\E{f(\te,(\mathbf{V},\mathbf{W}))~|~\te~}$ as an integral over the possible values of $\mathbf{V}$, the normalization factor exactly 
cancels the term $( \int_0^1 \te(u) du)^k$ in the numerator, and so the change of measure (\ref{eqn:changemeas}) yields 
\begin{align}
\E{f(\te, \mathcal P)~|~|\mathcal P |=k}
& = \frac{\E{\int_0^1 \cdots \int_0^1 f(\te,(\mathbf{u},\mathbf{W})) \te(u_1) \ldots \te(u_k) du_1\ldots du_k 
\cdot \exp \big(- \int_0^1 \te(u) du \big) }}
{\E{\big( \int_0^1 \te(u) du \big)^k \cdot \exp \big(- \int_0^1 \te(u) du \big) }} \notag \\
& = \frac{\E{\int_0^1 \int_0^1 \cdots \int_0^1 f(\be,(\mathbf{u},\mathbf{W})) \be(u_1) \be(u_2) \ldots \be(u_k) du_1 du_2 \ldots du_k}}{\E{\big(\int_0^1 \be(u) du\big)^k}} \notag \\
& = \frac{\E{f(\be,(\mathbf{U},\mathbf{W})) \be(U_1) \be(U_2) \ldots \be(U_k)}}{\E{\be(U_1)\be(U_2)\ldots \be(U_k)}}, \label{eqn:thiszero}
\end{align}
where $\mathbf{U}=(U_1,\ldots,U_k)$ and $U_1,\ldots,U_k$ are independent $\mathrm{U}[0,1]$ random variables. 

Informally, the preceding calculation can be interpreted as saying 
that conditional on $|\mathcal P| = k$, the probability of seeing a given excursion $\be$ is proportional to $(\int_{0}^1 \be(u) du)^k$, and that this bias can be captured by choosing $k$ height-biased leaves of 
the conditioned tree (or, equivalently, points of the conditioned excursion). We next derive the consequences of this fact for distribution of the lengths in the core. 

\subsection{The subtree spanned by the height-biased leaves}
Recall that to obtain the core from the excursion $2 \te$ and the points $\mathcal P$ we first form the subtree $T_C(2\te,\mathcal P)$ of $\mathcal T_{2\te}$ which is spanned by the points
$\{x:\xi = (x,y) \in \mathcal P\} \cup \{r(x):\xi=(x,y) \in \mathcal P\}$, and then identify $x$ and $r(x)$ for each $\xi \in \mathcal P$. 
This is depicted in Figure~\ref{fig:contkernel}. 
We remark that $T_C(2\te,\mathcal P)$ is a subtree of the reduced tree $T_R(2\te,\mathcal P)$, since 
$T_R(2\te,\mathcal P)$ is the subtree of $\mathcal T_{2 \te}$ which is spanned by the root and the points in $\{x:\xi=(x,y) \in \mathcal P\}$, and each point $r(x)$ (or rather its equivalence class) is 
on the path from the root to $x$.

Conditional on $|\mathcal P|=k$, the tree $T_R(2\te,\mathcal P)$ consists of $2k - 1$ line-segments.  If $k \ge 2$, the core is obtained from them as follows.  First sample the $k$ \emph{path-points} corresponding to the leaves (this corresponds to the choice of the points $W_1,\ldots,W_k$ above).   The $2(k-1)$ vertices of the core are then precisely the branch-points which were already present and the path-points which have just been sampled, less whichever of these points happens to be closest to the root.  (This can be either a branch-point or a path-point.)  Now throw away the line-segment closest to the root (this gives $T_C(2\te, \mathcal P)$) and make the vertex-identifications.  This yields a core $C(2\te, \mathcal P)$ which has precisely $3(k-1)$ edges.

If $k = 1$, the core is obtained from the subtree of the tilted tree spanned by the root and the single height-biased leaf by sampling the path-point, throwing away the segment closest to the root and making the vertex-identification.

In either case, we will find it helpful here to think of the $2k-1$ line-segments which make up our reduced tree as a combinatorial tree with edge-lengths rather than as a real tree. This tree with edge-lengths has a certain \emph{tree-shape} which we will find it convenient to represent as a labeled binary combinatorial tree.  Label the leaves $1,2,\ldots,k$ arbitrarily.  Now label internal vertices by the concatenation of all of the labels of their descendants which are leaves.  We do not label the root.  Write $L_v$ for the length of the unique shortest line-segment which joins the vertex labeled $v$ to another vertex nearer the root. Write $T$ for the set of vertices (vertex-labels) of this tree, 
excluding the root. Since the edges of the tree can be derived from the vertex labels, we will refer to $T$ as the tree-shape. In the following, we write $w\preceq v$ to denote that $w$ is on the path between $v$ and the root, including $v$ 
but not the root. 

\begin{lem} \label{lem:treeshapelengths}
Let $k \ge 1$.  For a tree-shape $t$ and edge-lengths $\ell_v, v \in t$, let $\bl = \{\ell_v, v \in t\}$.  The joint density of the tree-shape $T$ and lengths $L_v, v \in T$ is
\[
\tilde{f}(t, \bl) \propto \left[\prod_{i=1}^k \Bigg(\sum_{w \preccurlyeq i} \ell_w \Bigg)\right] \cdot \Bigg(\sum_{v \in t} \ell_v\Bigg) \cdot \exp \Bigg( -\frac{1}{2} \Bigg( \sum_{v \in t} \ell_v \Bigg)^2 \Bigg).
\]
\end{lem}

\begin{proof}
In order to see this, recall from page \pageref{twostage} that if $| \mathcal P |=k$ then the $k$ leaves are at heights given by $\te(V_1), \te(V_2), \ldots, \te(V_k)$ where, given $\te$, $V_1, V_2, \ldots, V_k$ are independent and identically distributed with density proportional to $\te(u)$.  From the excursion $\te$ and the values $V_1, V_2, \ldots, V_k$, it is possible to read off the tree-shape $T$ and the lengths $L_v, v \in T$.  
We will write $T = T(\te,\mathbf V)$.

Given a particular tree-shape, if $v = i_1i_2\ldots i_r$ then the vertex $v$ is at height
\[
\min \{\te(u) :
V_{i_1} \wedge V_{i_2} \wedge \dots \wedge V_{i_r} \le u \le V_{i_1} \vee V_{i_2} \vee \dots \vee V_{i_r} \}.
\]
Thus, if $v$ has parent $w = v i_{r+1}i_{r+2} \ldots i_{r+s}$ for some $i_{r+1}, i_{r+2},  \ldots, i_{r+s}$ all different from $i_1, \ldots, i_r$, then
\begin{align*}
L_v = & \min \{\te(u) :
V_{i_1} \wedge V_{i_2} \wedge \dots \wedge V_{i_r} \le u \le V_{i_1} \vee V_{i_2} \vee \dots \vee V_{i_{r}} \} \\
& - \min \{\te(u) :
V_{i_1} \wedge V_{i_2} \wedge \dots \wedge V_{i_{r+s}} \le u \le V_{i_1} \vee V_{i_2} \vee \dots \vee V_{i_{r+s}} \}.
\end{align*}
In order to make the dependence on $\te$ and $\mathbf{V} = (V_1, V_2, \ldots, V_k)$ clear, write $L_v = L_v(\te, \mathbf V)$.
So, using the tower law for conditional expectations and the change of measure as we did in (\ref{eqn:thiszero}), we obtain 
\begin{equation}
 \Prob{L_v(\te, \mathbf V) > x_v, v \in T(\te, \mathbf V)~|~| \mathcal P |=k}
= \frac{\E{\I{L_v(\be,\mathbf U) > x_v, v \in T(\be,\mathbf U)} \be(U_1) \be(U_2) \ldots \be(U_k)}}{\E{\be(U_1)\be(U_2)\ldots \be(U_k)}}, \label{eqn:thisone}
\end{equation}
where $\mathbf U = (U_1, U_2, \ldots, U_k)$ and $U_1, U_2, \ldots, U_k$ are independent $\mathrm{U}[0,1]$ random variables.
Note that $T(\be,\mathbf U)$ is then the tree-shape of the subtree of a standard Brownian CRT spanned by $k$ uniform leaves, and $\{L_v(\be, \mathbf U), v \in T(\be, \mathbf U)\}$ are its lengths.
It follows from equation (13) of Aldous~\cite{aldous91crt2} that for $T(\be,\mathbf U)$, 
the tree-shape and lengths have joint density
\begin{equation} \label{eqn:aldousdens}
f(t, \bl) = \Bigg(\sum_{v \in t} \ell_v \Bigg) \cdot \exp\left( -\frac{1}{2} \left(\sum_{v \in t} \ell_v \right)^2 \right).
\end{equation}
In particular, $T(\be, \mathbf U)$ is uniform on all possible tree-shapes and the lengths 
$\{L_v(\be, \mathbf U), v \in T(\be, \mathbf U)\}$
have an exchangeable distribution.  Moreover, for $1 \leq i \leq k$,
\[
\be(U_i) = \sum_{w \preccurlyeq i} L_w(\be, \mathbf U).
\]
Then from (\ref{eqn:thisone}), writing the expectations as integrals over the density and differentiating, we obtain the claimed result.
\end{proof}

\noindent \textbf{Remark.} Given this density representation, a natural hope would be that different values of $k$ could be coupled to obtain an increasing family of weighted trees $\{(T_k,\{L_v,v \in T_k\})\}_{k=1}^{\infty}$ such that 
for each $k$, $T_k$ and $\{L_v,v \in T_k\}$ have joint distribution given by the density in Lemma \ref{lem:treeshapelengths}. However, it is possible to check by hand that for the smallest 
non-trivial case, $k=4$, the distribution on tree shapes induced by the density in Lemma \ref{lem:treeshapelengths} is not uniform, and so the most naive strategy for accomplishing 
such a coupling (start from 
an increasing sequence of uniform leaf-labeled binary trees -- Catalan trees -- and then augment with random edge lengths) is unsurprisingly doomed to failure.

\subsection{Adding the points for identification: the pre-core}\label{precore}
Now consider the tree $T_R(2\te, \mathcal P)$ additionally marked with the path-points.  This yields a new tree with edge-lengths which will be important in the sequel, so we will call it the \emph{pre-core} (see Figure~\ref{fig:precore}, and compare with Figure~\ref{fig:reduced_tree}).  In particular, the pre-core consists of $3k-1$ line-segments whose lengths we will now describe.  

\begin{figure}[htb]
\centering
\begin{picture}(420,130)
\put(0,0){\includegraphics[scale=.7]{reduced_tree}}
\put(280,0){\includegraphics[scale=.7]{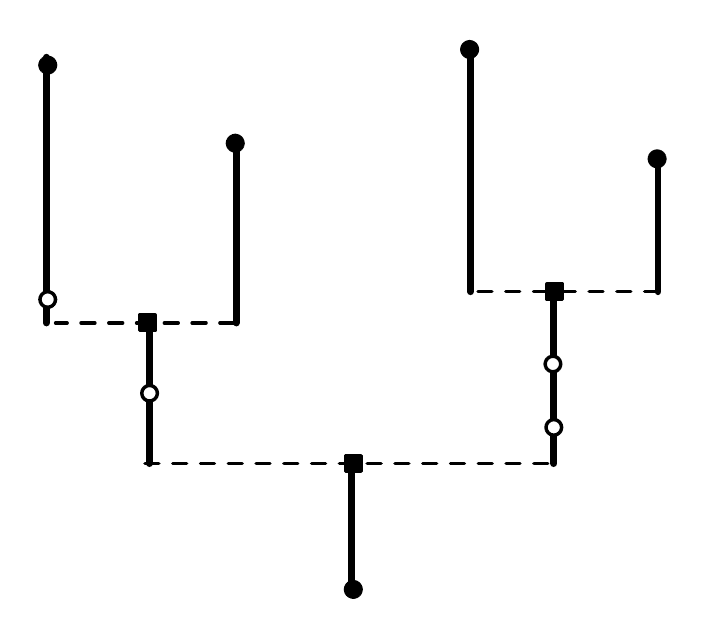}}
\put(28,57){$a$}
\put(63,45){$b$}
\put(121,49){$c$}
\put(147,33){$d$}
\put(288,120){$A$}
\put(330,100){$B$}
\put(370,123){$C$}
\put(409,100){$D$}
\put(303,66){$1$}
\put(350,40){$2$}
\put(385,72){$3$}
\put(280,63){$a$}
\put(302,45){$b$}
\put(384,53){$c$}
\put(395,40){$d$}
\end{picture}
\caption{\label{fig:precore}An excursion $h$ and the pre-core corresponding to the pointset $\mathcal Q=\{a, b, c, d\}$ (which has size $k = 4$). The pre-core is a combinatorial tree with edge-lengths.  It has $3k$ vertices: the root, the leaves, the path-points $a,b,c,d$ and the branch-points $1,2,3$. The dashed lines have zero length. }
\end{figure}

\begin{lem} \label{lem:pre-corelengths}
Suppose $k \ge 1$.  The lengths of the $3k-1$ line-segments in the pre-core have an exchangeable distribution.  With an arbitrary labeling, write $M_1, M_2, \ldots, M_{3k-1}$ for these lengths; then their joint density is proportional to
\begin{equation}\label{eq:pre-corelengths}
\left(\sum_{i=1}^{3k-1} m_i \right) \cdot \exp \Bigg( -\frac{1}{2} \left(\sum_{i=1}^{3k-1} m_i \right)^2 \Bigg).
\end{equation}
\end{lem}
\begin{proof}
Consider the locations of the marks.  The mark corresponding to leaf $i$ is uniform on the path from the root to the leaf $i$.  In a tree with tree-shape $t$ and lengths $\ell_v, v \in t$, this path has length $\sum_{w \preccurlyeq i} \ell_w$. For each leaf $1 \leq i \leq k$, let $w_i \preccurlyeq i$ be the vertex of $t$ closest to the root such that the path-point corresponding to $i$ lies between $w_i$ and the root.
The tree-shape, the lengths of the $2k-1$ edges of the tree and the vertices $w_i, 1 \le i \le k$ below which the uniform random variables corresponding to leaves $1, 2, \ldots, k$ fall have joint density
\begin{equation} \label{eqn:density}
\tilde{f}(t,\bl) \cdot \prod_{i=1}^k \frac{\ell_{w_i}}{\sum_{w \preccurlyeq i} \ell_w}
\propto \left(\prod_{i=1}^k \ell_{w_i}\right) \cdot \left(\sum_{v \in t} \ell_v \right) \cdot \exp\Bigg( -\frac{1}{2} \left(\sum_{v \in t} \ell_v \right)^2 \Bigg).
\end{equation}
Given that the uniform random variable corresponding to leaf $i$ falls in the edge below vertex $w_i$, the length $\ell_{w_i}$ gets split at a uniform point.  More generally, if $r$ uniforms fall in a particular edge below a vertex $w$, of length $\ell$, that edge gets split with an independent $\Dir(1,1, \ldots, 1)$ random variable, where the Dirichlet has $r+1$ co-ordinates.  Write $M_{w}^{1}, M_{w}^{2}, \ldots, M_{w}^{r+1}$ for the resulting lengths.  Note that the joint density of these lengths is then $\ell^{-r}$.  It follows that, whenever we split an edge of length $\ell$ into $r+1$ pieces, the density of the resulting pieces exactly cancels the factor of $\ell^r$ in (\ref{eqn:density}). Let $r(w)=|\{1\le i\le k: w_i=w\}|$ be the number path-points falling in the edge below $w$. Then by a change of variable (still conditional on the uniform random variable corresponding to leaf $i$ falling in the edge below $w_i$, for $1 \le i \le k$), the lengths $ M_w^{1}, M_w^{2}, \ldots, M_w^{r(w)+1}, w \in T$ have joint density proportional to
\[
\Bigg(\sum_{w \in T} \sum_{j=1}^{r(w) + 1} m_w^j \Bigg) \cdot \exp \left(-\frac{1}{2} \Bigg(\sum_{w \in T} \sum_{j=1}^{r(w) + 1} m_w^j \Bigg)^2 \right).
\]
Since this is symmetric in the variables $m_w^1, m_w^2, \ldots, m_w^{r(w) + 1}, w \in T$, we may take an arbitrary relabeling of the lengths.  The lengths, now labeled $M_1, M_2, \ldots, M_{3k-1}$, then have joint density proportional to (\ref{eq:pre-corelengths}), as required.
\end{proof}

\subsection{The lengths after identifications: the core}
Now recall that  we obtain the core from the pre-core by chopping off the line-segment closest to the root and making the vertex identifications.  We now know that the line-segments involved have an exchangeable distribution and so, in particular, the exact labeling is unimportant.  

\begin{lem} \label{lem:lengths}
\textsc{Complex components.} Suppose $k \ge 2$ and let $m = 3(k-1)$.  Write $B$ for the length of the line-segment closest to the root, write $C^*$ for the length of the core-edge to which it attaches and write $C_1, C_2, \ldots, C_{m-1}$ for the lengths of the other core edges.  Then
\[
(C_1, C_2, \ldots, C_{m-1}, C^*) \equidist \sqrt{\Gamma} (Z_1, Z_2, \ldots, Z_{m-1}, Z_m),
\]
where $(Z_1, Z_2, \ldots, Z_{m-1}, Z_m) \sim \Dir(1,1,\ldots,1,2)$ is independent of $\Gamma\sim \mathrm{Gamma}(\frac{m+1}2,\frac 12)$. The random variable $B$ depends on $C_1, C_2, \ldots, C_{m-1}, C^*$ only through their sum.  Conditional on $C_1+ C_2 + \cdots + C_{m-1} + C^* = w$, $B$ has density proportional to
\[
(w+ b) \exp \Big( -\frac{1}{2}\left[(w+b)^2 - w^2\right] \Big).
\]
\textsc{Unicyclic components.} Suppose that $k = 1$.  Write $C$ for the length of the cycle and $B$ for the length of the line-segment attaching it to the root.  Then
\begin{equation} \label{eqn:lollipoplengths}
(C,B) \equidist  \sqrt{\Gamma} (U, 1-U),
\end{equation}
where $\Gamma \sim \Ga(\frac 32,\frac 12)$ and $U \sim \mathrm{U}[0,1]$ are independent.
\end{lem}

\begin{proof}
Suppose first that $k \ge 2$.  One of the $3k-1$ edges $M_1, M_2, \ldots, M_{3k-1}$ is closest to the root.  Since the distribution of these random variables is exchangeable we can, without loss of generality, assume that $B = M_{3k-1}$.  Of the remaining lengths $M_1, M_2, \ldots, M_{3k-2}$, all but the two which (after vertex-identifications) are incident to the discarded length straightforwardly become edges of the core.  Once again without loss of generality, we can take $C_i = M_i$ for $1 \le i \le 3k-4$.  The two edges which are incident to the discarded length become a single core-edge, of length $C^* = M_{3k-3} + M_{3k-2}$.

We next make a straightforward change of variables: for $1 \le i \le 3k-3$, let
\[
V_i = \frac{M_i}{M_1 + \cdots + M_{3k-2}}.
\]
Let $W =  M_1 + \cdots + M_{3k-2}$.  Then the joint density of $V_1, V_2, \ldots, V_{3k-3}, W$ and $B = M_{3k-1}$ is easily shown to be
\[
w^{3k-3}(w+b) \exp \Big(-\frac{1}{2}(w + b)^2 \Big).
\]
This proves that $V_1, V_2, \ldots, V_{3k-3}$ are independent of $W$ and $B$ and that
\[
(V_1, V_2, \ldots, V_{3k-3}, 1 - V_1 - \cdots - V_{3k-3}) \sim \Dir(1,1,\ldots,1).
\]
Moreover it proves that $W$ has density proportional to $w^{3k-3} \exp\left(-\frac{1}{2}w^2 \right)$ and that, conditional on $W = w$, $B$ has density proportional to
\[
(w+ b) \exp \Big( -\frac{1}{2}\left[(w+b)^2 - w^2\right] \Big).
\]
The result follows by recalling that $m = 3k-3$ and the standard property of Dirichlet random variables that if $(A_1, A_2, \ldots, A_r) \sim \Dir(\alpha_1, \alpha_2, \ldots, \alpha_r)$ then $(A_1, A_2, \ldots, A_{r-2}, A_{r-1} + A_r) \sim \Dir(\alpha_1, \alpha_2, \ldots, \alpha_{r-2}, \alpha_{r-1} + \alpha_r)$.

In the case $k = 1$, a similar calculation shows that the two lengths $(M_1,M_2)$ which make up the pre-core satisfy
\begin{equation}
(M_1, M_2) \equidist \sqrt{\Gamma} (U, 1-U),
\end{equation}
where $\Gamma \sim \Ga(3/2,1/2)$ and $U \sim \mathrm{U}[0,1]$ independently.  Since these lengths are exchangeable, we can arbitrarily declare the first to be the length of the cycle and the second to be the length of the segment attaching it to the core.
\end{proof}

The results stated in the following lemma are straightforward and may be found, for example, in \citet{bertoingoldschmidt04coagfrag}.

\begin{lem} \label{lem:dir}
Suppose that $(Y_1, Y_2, \ldots, Y_m) \sim \Dir(1,1,\ldots,1)$.  Let $Y^*$ be a size-biased pick from this vector, and relabel the other co-ordinates arbitrarily $Y^*_1, Y^*_2, \ldots, Y^*_{m-1}$.  Then 
\[
(Y^*_1, Y^*_2, \ldots, Y^*_{m-1}, Y^*) \sim \Dir(1,1,\ldots,1,2).
\]
Moreover, if $U$ is an independent $\mathrm{U}[0,1]$ random variable,
\[
(Y^*_1, Y^*_2, \ldots, Y^*_{m-1}, Y^*U, Y^*(1-U)) \sim \Dir(1,1,\ldots,1).
\]
\end{lem}

The distributional identity (\ref{eqn:1}) follows straightforwardly from the first part of Lemma~\ref{lem:lengths}. 
Furthermore, by appeal to the finite-$n$ picture, it is obvious that the point at which the path from the root attaches to the core is a uniform point on the core. It follows that in the limit, the core edge to which the root attaches is a size-biased pick 
from among the core-edges, and attaches to a uniform point along this edge. This is exactly the procedure described in Lemma \ref{lem:dir}, and so 
the identity (\ref{eqn:2}) follows from Lemmas \ref{lem:lengths} and \ref{lem:dir}.  The identity (\ref{eqn:4}) follows from (\ref{eqn:lollipoplengths}) using the fact that the density of $\sqrt{\Gamma}$ is proportional to $x^2 e^{-x^2/2}$.  This identifies it as the size-biased Rayleigh distribution, written $R^*$. Finally, $U R^*$ has the same distribution as the modulus of a $\mathrm{Normal}(0,1)$ random variable (see p.121 of \citet{evanspitmanwinter06rayleigh}). \newline

\noindent \textbf{Remark.} Lemmas \ref{lem:lengths} and \ref{lem:dir} in fact tell us more than we need: they also
specify the distribution of the line-segment which originally attached the root to the
core.  In the case $k \ge 2$, this turns out to have the distribution of the time until the next point in an inhomogeneous Poisson process with instantaneous rate $t$ at time $t$, given that the first point is at $C(0)$.  (This is \emph{not} true in the case $k = 1$ where, as we will see in the next section, in order to have a stick-breaking construction we have to start from the core \emph{and} the edge attaching it to the root. 
We alluded to this somewhat mysterious requirement earlier in the paper.)
Moreover, this line-segment attaches in a uniform position on the core.  These facts will be useful in the next section.

\section{The stick-breaking construction of a limit component} \label{subsec:recursive}

We now prove the extension of the stick-breaking construction of the Brownian CRT which we stated in Theorem~\ref{thm:recursiveconstruction}.  Fix a number $k \geq 1$ of surplus edges.  We have already described the joint distribution of the $3k-1$ line-segments in the pre-core.  In order to create the whole metric space, it suffices to ``decorate'' the pre-core with the other parts of the tilted tree, and then make the right vertex-identifications.  As mentioned in the introduction, we use Aldous' notion of \emph{random finite-dimensional distributions} (see \cite{aldous93crt3}).  By Theorem 3 of \cite{aldous93crt3}, the distribution of a continuum random tree coded by an excursion is determined by the distributions of the sequence of finite subtrees obtained by successively sampling independent uniform points in the tree.  We will use this idea on our tilted trees.

It is perhaps surprising that the construction should be so similar to that of the Brownian CRT, considering that we start from a tilted excursion.  The key observation (proved below) is that the effect of tilting the tree can be felt entirely in the total length of the $k$ special branches which are used to create the pre-core or, equivalently, in the total length of the $3k-1$ line-segments which make up the pre-core.  

Once again, we need some notation and it is convenient to think again of tree-shapes and lengths.  The tree-shape of the pre-core is not a binary tree because the path-points are vertices of degree 2.  So we need an elaboration of our labeling-scheme.  The pre-core has $3k$ vertices which we will label as follows.  Label the leaves by $1,2,\ldots,k$ (in some arbitrary order).  Then label recursively from the leaves towards the root, which is, itself, left unlabeled.  For an internal vertex of degree three, label it by the concatenation of the labels of its two children.  For an internal vertex of degree two (i.e.\ a path-point) give it the label of its child, but with the label of the leaf whose path-point it is repeated.  Write $T^{(0)}$ for this tree-shape (where, as in the case of $T_R(2 \te, \mathcal P)$, the root is excluded), and for $v \in T^{(0)}$ write
$L^{(0)}_v$ for the length of the unique shortest line-segment which joins $v$ to another vertex nearer the root.  Then by Lemma~\ref{lem:pre-corelengths}, given $T^{(0)} = t$, the joint density of $L_v^{(0)}, v \in T^{(0)}$ is proportional to
\[
\left(\sum_{v \in t} \ell_v \right) \cdot \exp \Bigg( - \frac{1}{2} \left( \sum_{v \in t} \ell_v \right)^2 \Bigg).
\]
(Note the similarity to the density (\ref{eqn:aldousdens}) of the edge lengths in the subtree of a Brownian CRT spanning a collection of uniform leaves.)

In Lemma~\ref{lem:decorating} we prove that, given these lengths, the decoration process (stick-breaking construction) works in the same way as it does in Aldous' construction of the Brownian CRT.  That is, we take the inhomogeneous Poisson process from time $\sum_{v \in T^{(0)}} L_v^{(0)}$ onwards, and for each inter-jump time (the first being measured from $\sum_{v \in T^{(0)}} L_v^{(0)}$), we add a line-segment of that length at a uniformly-chosen point on the structure already created.

Starting from the pre-core edges, if we sample a uniform point from the tilted tree, it corresponds to a branch which attaches somewhere on the pre-core, splitting one of the pre-core lengths in two.  Sampling further points similarly causes the splitting in two of lengths already present.  Let $T^{(n)}$ be the tree-shape obtained by taking the pre-core and sampling $n$ additional uniform points of the excursion, with labeling as before.  Let $L_v^{(n)}$ be the corresponding lengths.

\begin{lem} \label{lem:decorating}
For $n \geq 1$, conditional on $T^{(n)} = t$, the joint density of the lengths $L_v^{(n)}$ is proportional to
\[
\left(\sum_{v \in t} \ell_v \right) \exp \left( - \frac{1}{2} \left( \sum_{v \in t} \ell_v \right)^2 \right).
\]
Moreover, this is the same as the joint distribution of lengths obtained by the stick-breaking construction proposed above.
\end{lem}

\begin{proof}
The first part of the proof is similar to that of Lemma \ref{lem:treeshapelengths}.  As there, the $k$ special leaves are at heights $\te(V_1), \te(V_2), \ldots, \te(V_k)$ where, given $\te$, $V_1, V_2, \ldots, V_k$ are independent and identically distributed with density proportional to $\te(u)$.  Let $W_1, W_2, \ldots, W_k$ be the $\mathrm{U}[0,1]$ random variables which give the positions along the paths to the root of the path-points corresponding to leaves $1,2,\ldots, k$ respectively.  $W_1, W_2, \ldots, W_k$ are mutually independent and independent of everything else.  Finally, let $U_1, U_2, \ldots$ be another sequence of independent $\mathrm{U}[0,1]$ random variables, which will generate the uniformly-chosen leaves, having heights $\te(U_1), \te(U_2), \ldots$.  For $n \geq 1$, write $T^{(n)} = T^{(n)}(\te, \mathbf{V}, \mathbf{W}; \mathbf{U})$ and, for $v \in T^{(n)}$, $L_v^{(n)} = L_v^{(n)}(\te, \mathbf{V}, \mathbf{W}; \mathbf{U})$ and note that these quantities can be calculated explicitly from the random ingredients $\te$, $\mathbf{V} = (V_1, V_2, \ldots, V_k)$, $\mathbf{W} = (W_1, W_2, \ldots, W_k)$, and $\mathbf{U} = (U_1, U_2, \ldots, U_n)$, although in a somewhat complicated way.  Using the change of measure,
\begin{multline*}
\Cprob{L_v^{(n)}(\te, \mathbf{V}, \mathbf{W}; \mathbf{U}) > x_v \ \forall\ v \in T^{(n)}(\te, \mathbf{V}, \mathbf{W}; \mathbf{U})}{|\mathcal P | = k}  \\
\propto \E{\I{L^{(n)}_v(\be, \mathbf{U}', \mathbf{W}; \mathbf{U}) > x_v \ \forall\ v \in T^{(n)}(\be, \mathbf{U}', \mathbf{W}; \mathbf{U})} \be(U_1') \be(U_2') \ldots \be(U_k')},
\end{multline*}
where $U_1', U_2', \ldots, U_k'$ are more independent $\mathrm{U}[0,1]$ random variables, independent of everything else.  Note that $T^{(n)}(\be, \mathbf{U}', \mathbf{W}; \mathbf{U})$ is just the tree-shape derived from picking $k+n$ uniform leaves and picking path-points for the first $k$ of them, and $L_v^{(n)}(\be, \mathbf{U}', \mathbf{W}; \mathbf{U})$ are the corresponding lengths.  The claimed joint density then follows from the same arguments as used in the proofs of Lemmas~\ref{lem:treeshapelengths} and~\ref{lem:pre-corelengths}.

The proof that this is the joint distribution given by the stick-breaking construction is identical to that of Lemma 21 in Aldous~\cite{aldous93crt3}.
\end{proof}
We now prove Theorem \ref{thm:recursiveconstruction} and thereby justify the third of our construction techniques. 
\begin{proof}[Proof of Theorem~\ref{thm:recursiveconstruction}]
Let $m=m(n)$ and $p=p(n)$ be such that $mn^{-2/3} \to 1$ and $pn \to 1$. Consider a probability space in which $m^{-1/2}G_m^p \to g(2 \te,\mathcal P)$ almost surely as $n \to \infty$; such a space exists by Theorem \ref{thm:clcrg} and Skorohod's 
representation theorem.  Theorem 7 of \cite{janson93birth} implies that for any $3$-regular kernel $K$ with fixed surplus $k\geq 2$ and with $t$ loops, 
\[
\p{ K(G_m^p) = K | G_m^p~\mbox{has surplus }k} \propto (1+o(1)) \Bigg(2^t\prod_{e \in E(K)} \mathrm{mult}(e)!\Bigg)^{-1},
\]
as $m \to \infty$.  Furthermore, by \cite{LuPiWi1994}, Theorem 4, all vertices of degree three in $K(G_m^p)$ are separated by distance of order $m^{1/2}$, so all such vertices remain distinct in the limit. 
This proves that the shape of the limiting kernel $K(2 \te,\mathcal P)$ has the claimed distribution.  

Next, the fact that the lengths in the kernel and rooted pre-core are as in Theorem \ref{thm:recursiveconstruction} follows from (\ref{eqn:1}), (\ref{eqn:2}) and (\ref{eqn:4}). Since $2\te$ almost surely encodes a compact real tree, it is also clear that the trees $T^{(m)}$ of 
Lemma \ref{lem:decorating} converge almost surely to the tree encoded by $2 \te$. Lemma 21 of \cite{Us1} says that a finite number of vertex identifications encoded by a fixed number of points under the contour process will not disrupt this convergence and so, 
by the validity of the method described in Section \ref{vitt}, we obtain almost sure convergence of the sequence of metric spaces created by the described process to a metric space with distribution $g(2 \te,\mathcal P)$. This completes the proof.
\end{proof}

\section{An urn process to analyze the stick-breaking construction}\label{sec:urns_distrib}

A careful analysis of the stick-breaking construction will enable us to prove Theorems~\ref{thm:limit_kernel} and~\ref{thm:unicyclic}. 

We assume that, as in Section~\ref{sec:length_core} the core has $m$ edges of lengths $L_1(0), L_2(0), \ldots,
L_m(0)$. Then we can think of our stick-breaking construction as a sort
of continuous Markovian balls-in-urns procedure acting on the lengths $L_1(n),
L_2(n), \ldots, L_m(n)$ representing the (continuous) quantities of $m$ different
colors that we have at step $n$ of this procedure, which we now describe. For each $n \ge 0$, we write $C(n) = \sum_{i=1}^m L_i(n)$ and, for $1\le i\le m$, we define the proportion $P_i(n)=L_i(n)/C(n)$.  

Given an inhomogeneous Poisson point process of instantaneous rate $t$
at time $t$ has a point at $c$, the density of time until the next
point is proportional to
\[
(a+c) \cdot \exp \Big( -\frac{1}{2}(a+c)^2 + \frac{1}{2}c^2 \Big).
\]
Suppose now we have already constructed
$L_1(n), L_2(n), \ldots, L_m(n)$ for some $n \geq 0$.  At step $n+1$, select an index $I(n+1)$ from $\{1, \dots, m\}$ so that
\[
\Cprob{I(n+1) = i}{P_1(n), P_2(n), \ldots, P_m(n)} = P_i(n).
\]
Then sample a random variable $A(n+1)$ such that, conditional on
$C(n) = c$, $A(n+1)$ has density proportional to
\[
(a+c) \cdot \exp \Big(-\frac{1}{2} (a+c)^2 + \frac{1}{2} c^2 \Big).
\]
Finally, set
\[
L_j(n+1)  =  \begin{cases}
         L_j(n) & \text{ if $j \neq I(n+1)$}, \\
         L_j(n) + A(n+1) & \text{ if $j = I(n+1)$},
\end{cases}
\]
and set $C(n+1)=C(n)+A(n+1)$.
In other words, add a quantity $A(n+1)$ to color $I(n+1)$ and increment $n$.  It is clear that this procedure describes precisely the dynamics of the process $(L_1(n), L_2(n), \ldots, L_m(n))_{n \ge 0}$. 

We now characterize the evolution of the total length $C(n)$.

\begin{lem} \label{lem:densities}
  For $c \geq 0$, the random variable $C(n)^2$ has a $\Ga((m+2n+1)/2, 1/2)$ distribution.  Moreover, for $n \geq 1$, 
\[
\frac{1}{C(n)} \left(C(n-1), A(n) \right) \sim \Dir(n+2m-1, 1),
\]
independently of $C(n) = C(n-1) + A(n)$.
\end{lem}

\begin{proof}
We proceed by induction.  For $n = 0$, the first statement is clear
from (\ref{eqn:1}).  Suppose now that $C(n-1)$ has the claimed distribution.  Then $C(n-1)$ and $A(n)$ have joint
density proportional to
\[
c^{m+2(n-1)} (a+c) \exp \Big( - \frac{1}{2} (a + c)^2 \Big).
\]
Write $V = C(n-1)/(C(n-1)+A(n))$ and $W = C(n-1)+ A(n)$.  A straightforward change of
variables gives that $V$ and $W$ have joint density proportional to $v^{m+2(n-1)} w^{m+2n} e^{-w^2/2}$.
Hence, $V$ and $W$ are independent and have the claimed distributions.
\end{proof}

The continuous urn model described above can be studied using an associated discrete urn process related to
\[
N_j(n) = 1 + 2 \sum_{k = 1}^n \I{I(k) = j},
\]
the number of branches corresponding to the $j$th core edge at step $n$, for $n \geq 0$, $1 \leq j \leq m$.  The following lemma is in the same spirit as Exercises 7.4.11 to 7.4.13 of \citet{PitmanStFl}.

\begin{lem}\label{lem:cont2disc}Conditional on $N_1(n), \dots, N_m(n)$,
\begin{equation}\label{eq:pdir}
(P_1(n), P_2(n), \ldots, P_m(n)) \sim\Dir(N_1(n), N_2(n), \ldots, N_m(n)).
\end{equation}
Furthermore, the process $(N_1(n),\dots, N_m(n))_{n\ge 0}$ evolves as the number of balls of $m$ different colors in a P\'olya's urn process started with one ball of each color and where, at each step, the ball picked is returned to the urn along with two extra balls of the same color.
\end{lem}	

In order to prove this, we first need to extend the first part of Lemma \ref{lem:dir}.  

\begin{lem} \label{lem:dirsb} Suppose that $(Y_1, Y_2, \ldots, Y_n)
  \sim \Dir(\alpha_1, \alpha_2, \ldots, \alpha_n)$.  Let
  $I$ be the index of a size-biased pick from this vector, i.e.\ $I$
  has conditional distribution $\Cprob{I = i}{Y_1, Y_2, \ldots, Y_n} = Y_i.$
Then 
\[
\Prob{I=i} = \frac{\alpha_i}{\sum_{j=1}^n \alpha_j}
\]
and, conditional on $I=i$, we have $(Y_1, Y_2, \ldots, Y_n) \sim \Dir(\alpha_1, \ldots, \alpha_{i-1}, \alpha_i + 1, \alpha_{i+1}, \ldots, \alpha_n).$
\end{lem}

\begin{proof}
  Let $G_1, G_2, \ldots, G_n$ be independent random variables such
  that $G_i \sim \mathrm{Gamma}(\alpha_i, 1)$ for $1 \le i \le n$.
  Write $G = \sum_{j=1}^n G_j$ and $G^{(i)} = G - G_i$.  Let $\Phi:
  \Delta_n \to \R^+$ be any non-negative measurable function.  Then
\[
\E{\Phi(Y_1, Y_2, \ldots, Y_n) \I{I=i}} = \E{\frac{G_i}{G} \Phi\left(\frac{G_1}{G}, \ldots, \frac{G_n}{G} \right)}.
\]
Note that $G$ is independent of $\left(\frac{G_1}{G}, \ldots, \frac{G_n}{G} \right)$ and so, since $\E{G} = \sum_{i=1}^n \alpha_i$,
\[
\E{G_i \Phi\left(\frac{G_1}{G}, \ldots, \frac{G_n}{G} \right)} = \E{G \cdot \frac{G_i}{G} \Phi\left(\frac{G_1}{G}, \ldots, \frac{G_n}{G} \right)} = \left(\sum_{i=1}^n \alpha_i\right) \E{\frac{G_i}{G} \Phi\left(\frac{G_1}{G}, \ldots, \frac{G_n}{G} \right)}.
\]
Integrating over the density of $G_i$, we see that
\begin{align*}
& \E{G_i \Phi\left(\frac{G_1}{G}, 
  \ldots, \frac{G_n}{G} \right)} \\
& = \E{ \int_0^{\infty} x \Phi \left(\frac{G_1}{x + G^{(i)}}, 
        \ldots, \frac{G_{i-1}}{x + G^{(i)}}, 
        \frac{x}{x + G^{(i)}}, 
        \frac{G_{i+1}}{x + G^{(i)}}, \ldots,  
        \frac{G_n}{x + G^{(i)}} \right)
        \frac{1}{\Gamma(\alpha_i)} x^{\alpha_i - 1} e^{-x} dx} \\
& = \frac{\Gamma(\alpha_i + 1)}{\Gamma(\alpha_i)}
\E{\Phi \left(\frac{G_1}{\gamma + G^{(i)}}, 
        \ldots, \frac{G_{i-1}}{\gamma + G^{(i)}}, 
        \frac{\gamma}{\gamma + G^{(i)}}, 
        \frac{G_{i+1}}{\gamma + G^{(i)}}, \ldots,  
        \frac{G_n}{\gamma + G^{(i)}} \right)},
\end{align*}
where $\gamma \sim \mathrm{Gamma}(\alpha_i + 1, 1)$, independently of
$G_j, j \neq i$.  The result follows since the proportions $G_j/(\gamma+G^{(i)})$, $j\ne i$ and $\gamma/(\gamma+G^{(i)})$ are independent of the total sum $\gamma+G^{(i)}$.
\end{proof} 

\begin{proof}[Proof of Lemma~\ref{lem:cont2disc}] We proceed by induction on $n$. It is clear that the distributional identity holds for $n
= 0$.  Suppose now that it also holds for $n = k$.  Then, by
Lemma~\ref{lem:dirsb},
\[
\Cprob{I(k+1) = i}{N_1(k), N_2(k), \ldots, N_m(k)} =
\frac{N_i(k)}{\sum_{j=1}^m N_j(k)}
\]
and, conditional on $N_1(k), N_2(k), \ldots, N_m(k)$ and $I(k+1)
= i$,
\[
(P_1(k), P_2(k), \ldots, P_m(k)) \sim
\Dir(N_1(k), \ldots, N_{i-1}(k), N_i(k) + 1,
N_{i+1}(k), \ldots, N_m(k)).
\]
Recall from Lemma~\ref{lem:densities} that
\[
\frac{1}{C(k+1)}(C(k), A(k+1)) \sim \Dir(m+2k+1,1),
\]
independently of $C(k+1)$.  So, conditional on $N_1(k), N_2(k), \ldots, N_m(k)$ and $I(k+1) = i$,
\[
(P_1(k+1), P_2(k+1), \ldots, P_m(k+1)) \sim \Dir(N_1(k), \ldots, N_{i-1}(k), N_i(k) + 2, N_{i+1}(k), \ldots, N_m(k)).
\]
The claimed results follow by induction on $n$.  
\end{proof}	

The proofs of Theorems \ref{thm:limit_kernel} and \ref{thm:unicyclic} rest on Lemmas~\ref{lem:convurn} and \ref{lem:convproc} below.

\begin{lem}\label{lem:convurn}
As $n \to \infty$,
\[
(P_1(n), P_2(n), \ldots, P_m(n)) \to (P_1, P_2, \ldots, P_m) \text{ a.s.,}
\]
where $(P_1, P_2, \ldots, P_m) \sim \Dir(\frac 12, \frac 12,
\ldots, \frac 12)$.
\end{lem}

\begin{proof}
It is straightforward to see that the process $(P_1(n), P_2(n), \ldots, P_m(n))_{n \ge 0}$
is a bounded martingale and so possesses an almost sure limit, $(P_1, P_2, \ldots, P_m)$.  So we
need only determine the distribution of the limit. We do so using the correspondence with the discrete urn process $(N_1(n),\dots, N_m(n))_{n\ge 0}$ given by Lemma~\ref{lem:cont2disc}.

By Lemma~\ref{lem:cont2disc}, $(N_1(n), N_2(n), \ldots, N_m(n))_{n \geq 0}$ is performing
P\'olya's urn scheme with $m$ colors where the ball picked is replaced along with two extra balls of the same color.  It is
standard (see, for example, Section VII.4 of \citet{FellerVol2} or Chapter V, Section 9.1 of \citet{AthreyaNey}) that the proportions of balls of each color converge almost
surely; indeed,
\[
\frac{1}{m + 2n}(N_1(n), N_2(n), \ldots, N_m(n)) \to (N_1, N_2, \ldots, N_m)
\]
almost surely, where $(N_1, N_2, \ldots, N_m) \sim \Dir(\frac 12, \frac 12, \ldots, \frac 12)$.  

Now let $(E_{j,k}, 1 \leq j
\leq m, k \geq 1)$ be i.i.d.\ standard
exponential random variables. Then, given $N_1(n), N_2(n), \ldots, N_m(n)$,
\begin{equation}\label{eq:pdist}
(P_1(n), P_2(n), \ldots, P_m(n))
\equidist \frac{1}{\sum_{j=1}^m \sum_{k=1}^{N_j(n)} E_{j,k}}
\Bigg(\sum_{k=1}^{N_1(n)} E_{1,k}, \sum_{k=1}^{N_2(n)} E_{2,k}, \ldots,
  \sum_{k=1}^{N_m(n)} E_{m,k} \Bigg).
\end{equation}
Thus, for each $i=1,\ldots,m$, using the fact that $P_i(n)$ and $N_i(n)$ both possess almost sure limits, we have 
\begin{align*}
\p{P_i \neq N_i} & = \sup_{\epsilon > 0}\p{\exists n_0,~\forall n \geq n_0, \left|P_i(n) - \frac{N_i(n)}{m+2n}\right|>\epsilon} \\
			 & \le \sup_{\epsilon > 0} \liminf_n \p{\left|P_i(n) - \frac{N_i(n)}{m+2n}\right|>\epsilon} \\
			 & = 0,
\end{align*}
where the second line follows from Fatou's lemma and the last equality follows from (\ref{eq:pdist}) and the fact that $N_1(n)+\ldots+N_m(n)=m+2n$. 
It follows that $\lim_{n \to \infty} (P_1(n),\ldots,P_m(n)) = (P_1, \ldots, P_m) = (N_1,\ldots,N_m)$ almost surely, which proves the lemma.
\end{proof}

\begin{lem} \label{lem:convproc} 
There exists a process $(X_1(n),\ldots,X_m(n))_{n \geq 0}$ such that for each $n$, 
conditional on the vectors $(N_1(n), N_2(n), \ldots, N_m(n))$ and $(X_1(n),\ldots,X_m(n))$, 
the sequence of additions to index $i$ up to time $n$ (including the initial length 
$L_i(0)$) is the sequence of the first $N_i(n)$ inter-jump times of an inhomogeneous
Poisson process of instantaneous rate $t$ at time $t$, all multiplied
by $X_i(n)$.  Moreover, under the above conditioning, 
these processes are independent for distinct $i$. Finally, 
$(X_1(n), X_2(n), \ldots, X_m(n)) \to (\sqrt{P_1}, \sqrt{P_2}, \ldots, \sqrt{P_m})$ almost surely as $n \to \infty$.
\end{lem}

We will need the following lemma on the lengths in the stick-breaking construction of the Brownian CRT.  

\begin{lem}\label{aldouslem}
Let $L_1, L_2, \ldots, L_{1+2k}$ be the lengths in the tree created by
the stick-breaking construction of the Brownian CRT up to its $k$th step, so that we have added $k$
branches to the tree.  Then
\[
(L_1, L_2, \ldots, L_{1+2k}) \equidist \sqrt{\Gamma} \cdot  (D_1, D_2, \ldots, D_{1+2k}),
\]
where $(D_1, D_2, \ldots, D_{1+2k}) \sim \Dir(1,1,\ldots,1)$ and
$\Gamma$ is independent with $\mathrm{Gamma}(k+1, 1/2)$ distribution.
\end{lem}
\begin{proof}
From Aldous~\cite{aldous91crt2}, we have that $L_1, L_2, \ldots, L_{1+2k}$ have joint density
\[
\left( \sum_{i=1}^{1+2k} \ell_i \right) \exp \Bigg( - \frac{1}{2} \left( \sum_{i=1}^{1+2k} \ell_i \right)^2 \Bigg).
\]
Let $W = L_1 + L_2 + \dots + L_{1+2k}$, $D_i = L_i / W$ for $1 \le i \le 2k$ and $D_{1 + 2k} = 1 - D_1 - D_2 - \dots - D_{2k}$.  Then these random variables have joint density proportional to $w^{2k+1} e^{-w^2/2}$.
The result follows.
\end{proof}

\begin{proof}[Proof of Lemma \ref{lem:convproc}]
Fix $n \geq 1$ and $1 \leq i \leq m$. We remark that we may think of $L_i(n)$ as composed of $N_i(n)$ segments 
of lengths 
\[L_i^1(n),\ldots,L_i^{N_i(n)}(n).
\]
These are the lengths of the line-segments which make up the subtree $T_i(n)$ corresponding to the $i$th core edge at step $n$ in the construction. Let $H_i(n)$ be the number of times that the $i$th core edge has been hit; since each addition creates two new line segments, we have $H_i(n)=(N_i(n)-1)/2$.
An argument using Lemma~\ref{lem:dir} shows that for all $n$,  conditionally on $N_i(n)$,
\begin{equation}\label{exactlyright}
\bigg(\frac{L_i^1(n)}{L_i(n)},\ldots,\frac{L_i^{N_i(n)}(n)}{L_i(n)}\bigg) \sim  \Dir(1,1,\ldots,1),
\end{equation}
and this vector is independent of $(L_1(n),\ldots,L_m(n))$ and from
$L_j^1(n)/L_j(n),\ldots,L_j^{N_j(n)}(n)/L_j(n)$,
for $j \neq i$. Since the elements of this vector are exchangeable, 
it follows immediately that the next segment to choose colour $i$ is equally likely to attach to any 
of the $N_i(n)$ segments. Thus, at all times $n$, the tree shape of the subtree $T_i(n)$ composed 
of segments corresponding to balls of colour $i$ is uniform over Catalan trees with $1+H_i(n)$ leaves (in fact, if we ignore the edge lengths, this is precisely R\'emy's algorithm to generate a uniform Catalan tree \cite{Remy1985}). 

Furthermore, by Lemma~\ref{aldouslem}, (\ref{exactlyright}) is exactly the right 
distribution for the {\em relative} lengths of edges in the tree $\mathcal A_{H_i(n)}$ created by running the stick-breaking construction of the Brownian CRT for $H_i(n)$ steps. 

Now, for each $1 \leq i \leq m$, let $(\lambda_i(n))_{n \ge 0}$ be an independent copy of the sequence of arrival times of an inhomogeneous 
Poisson process of instantaneous rate $t$ at time $t$ (with $\lambda_i(0)>0$ the first arrival time), and let 
$$(X_i(n))_{n \ge 0} = \left(\frac{L_i(n)}{\lambda_i(H_i(n))}\right)_{n \ge 0}.$$ 
Also, for each $1 \le i \le m$ and each $k \ge 0$, let 
$t_i(k) = \min\{n:H_i(n) = k\},$
so $(t_i(k))_{k \ge 0}$ is the sequence of times at which $N_i(n)$ increases.
Write $T_i^*(k)$ for the tree $T_i(t_i(k))$, above, rescaled to have total length $\lambda_i(k)$. It follows that $(T_i^*(k))_{k \ge 0}$ is a copy of the sequence of trees $\{\mathcal A_k\}_{k \ge 0}$ 
created by the stick-breaking construction of the Brownian CRT, and is 
independent of $(L_1(n),\ldots,L_m(n))_{n \ge 0}$, $(N_1(n),\ldots,N_m(n))_{n \ge 0}$, and $(T_j^*(k))_{k \ge 0}$ for $j \neq i$. 
The first two claims in the lemma are then immediate. 

We have $X_i(n)=L_i(n)/\lambda_i(H_i(n))$.  We can write this as
\begin{equation} \label{eqn:split}
X_i(n)=\frac{\sqrt{N_i(n) - 1}}{\lambda_i(H_i(n))} \times P_i(n) \times \frac{C(n)}{\sqrt{m+2n+1}} \times \sqrt{\frac{m+2n+1}{N_i(n)-1}}.
\end{equation}
From Lemma~\ref{lem:convurn}, we have $P_i(n) \to P_i$ almost surely.
Recall from Lemma~\ref{aldouslem} that $\lambda_i(H_i(n))^2 \sim \Ga(H_i(n)+1,\frac 12)$. 
Since $\lambda_i(k) \nearrow \infty$ as $k \to \infty$ and $H_i(n),N_i(n) \to \infty$ almost surely, it follows that 
\[
\frac{\lambda_i(H_i(n))}{\sqrt{N_i(n)-1}} \sim \frac{\lambda_i(H_i(n))}{\sqrt{2 (H_i(n)+1)}} \to 1~\text{almost surely}.
\]
Similarly, from Lemma~\ref{lem:densities} we have $C(n)^2 \sim \Ga((m+2n+1)/2,1/2)$ and $C(n) \nearrow \infty$ almost surely and so
\[
\frac{C(n)}{\sqrt{m+2n+1}} \to 1~\text{almost surely}.
\]
In the proof of Lemma~\ref{lem:convurn}, we showed that
\[
\frac{N_i(n)}{n+2m} \to P_i~\text{almost surely.}
\]
Putting these facts together in (\ref{eqn:split}), it follows that $X_i(n) \to \sqrt{P_i}$ almost surely, completing the proof. 
\end{proof}

We can summarize/rephrase the results of Lemmas \ref{lem:convurn} and \ref{lem:convproc} as the following counterpart of the classical limit result for urn models \cite{BlKe1964, Freedman1965} (which is usually proved using de Finetti's theorem \cite{Finetti1931}).  We do not know of a pre-existing reference for this result in the literature.

\begin{thm} \label{thm:summary}
Consider the balls-in-urns model described at the beginning of the section, with quantities $L_1(n), L_2(n), \ldots, L_m(n)$ of the $m$ different colors present at step $n$.  The proportions of the different colors present converge almost surely to $(P_1, P_2, \ldots, P_m) \sim \Dir(\frac 12,\frac 12,\ldots,\frac 12)$.  Moreover, conditional on $(P_1, P_2, \ldots, P_m)$, the indices $I(1), I(2), \ldots$ are independent and identically distributed with
\[
\Cprob{I(1) = i }{ P_1, P_2, \ldots, P_m} = P_i, \quad 1 \le i \le m.
\]
Finally, conditional on $(P_1, P_2, \ldots, P_m)$ the sequence of additions to color $i$ is the sequence of inter-jump times of an inhomogeneous Poisson process with rate $t$ at time $t$, rescaled by $\sqrt{P_i}$, for $1 \le i \le m$.
\end{thm}

\noindent \textbf{Remark.} The gamma duplication formula (\ref{eqn:gammadup}) played a central role in the proof of Proposition~\ref{prop:raydir}. Equation (65) of \citet{pitman99brownian} states the following generalization.  Suppose that for $r > 0$ and $s = 1,2,\ldots$, $A \sim \Ga(r,\frac 12)$ and $B \sim \Ga(r+s-\frac 12,\frac 12)$.  Suppose that $J_{r,s}$ has distribution
\[
\Prob{J_{r,s} = j} = \frac{(2s-j-1)! (2r)_{j-1}}{(s-j)! (j-1)! 2^{2s-j-1}(r+ \frac{1}{2})_{s-1}},
\]
where $(x)_n = x(x+1)(x+2)\cdots(x+n-1) = \Gamma(x+n)/\Gamma(x)$.  Finally, conditional on $J_{r,s}$, let $C$ have $\Ga(2r+J_{s,r}-1,1)$ distribution.  Then
\begin{equation} \label{eqn:gammadup2}
AB \equidist C^2.
\end{equation}
Suppose now that $(P_1, P_2, \ldots, P_m) \sim \Dir(\frac 12,\frac 12,\ldots,\frac 12)$ and let $(M_1(n), M_2(n), \ldots, M_m(n)) \sim \mathrm{Multinomial}(n;P_1, P_2, \ldots, P_m)$.  For $1 \leq i \leq m$, let $G_i(n) \sim \Ga(1+ M_i(n), \frac 12)$ independently.  Let $\Gamma_n \sim \Ga((m+2n+1)/2,1/2)$.  Let $N_1(n), N_2(n), \ldots, N_m(n)$ be the number of balls at step $n$ of P\'olya's urn model started with one ball of each color and such that each ball picked is replaced along with two more of the same color.  Finally, let $(P_1(n), P_2(n), \ldots, P_m(n)) \sim \Dir(N_1(n), N_2(n), \ldots, N_m(n))$.  Then as a consequence of Lemma~\ref{aldouslem} and Theorem~\ref{thm:summary}, we have
\[
(\sqrt{P_1 G_1(n)}, \sqrt{P_2 G_2(n)}, \ldots, \sqrt{P_mG_m(n)}) \equidist \sqrt{\Gamma_n} (P_1(n), P_2(n), \ldots, P_m(n)).
\]
It seems likely that some version of (\ref{eqn:gammadup2}) for appropriate values of $r$ and $s$ is hidden in this distributional relationship.
\newline

We can, at last, complete the proofs of Theorems~\ref{thm:limit_kernel} and~\ref{thm:unicyclic}.

\begin{proof}[Proof of Theorem~\ref{thm:limit_kernel}]
Theorem \ref{thm:limit_kernel} (a) was already proved in the course of proving Theorem \ref{thm:recursiveconstruction}. Part (b) follows from Theorem \ref{thm:recursiveconstruction} as the construction procedure in that theorem almost surely places no mass at the vertices of the kernel. Part (c) is immediate from Lemma \ref{lem:convurn}, and (d) and (e) are immediate from Lemma~\ref{lem:convproc} and Theorem~\ref{thm:aldousrecursive}. 
\end{proof}

\begin{proof}[Proof of Theorem~\ref{thm:unicyclic}]
Theorem~\ref{thm:unicyclic} (a) is precisely the identity (\ref{eqn:4}), and (b) follows from the second part of Lemma \ref{lem:lengths}, Lemma \ref{lem:convproc} and Theorem~\ref{thm:aldousrecursive}. 
\end{proof}

Finally, as mentioned earlier, the validity of Procedure 1 is a consequence of Theorems~\ref{thm:limit_kernel} and~\ref{thm:unicyclic}.

\setlength{\bibsep}{.3em}
\bibliographystyle{plainnat}
\bibliography{bib_dlcrg}

\hspace{1cm}

\begin{tabbing}
\textsc{Louigi Addario-Berry} \\                    
Department of Mathematics and Statistics\\ 
McGill University\\                        
Burnside Hall, Room 1219\\                
805 Sherbrooke W. \\
Montr\a'eal, QC \\
H3A 2K6, Canada \\
\texttt{louigi@gmail.com}\\
\texttt{http://www.math.mcgill.ca/louigi/} \\ 
\\
\textsc{Nicolas Broutin} \\
Projet Algorithms \\ 
INRIA Rocquencourt  \\
78153 Le Chesnay \\
France \\
\texttt{nbrout@cs.mcgill.ca} \\
\texttt{http://algo.inria.fr/broutin/} \\
\\
\textsc{Christina Goldschmidt}\\
Department of Statistics \\
University of Warwick\\
Coventry\\
CV4 7AL\\
UK\\
\texttt{C.A.Goldschmidt@warwick.ac.uk} \\
\texttt{http://www.warwick.ac.uk/\~{}stsiac}
\end{tabbing}

\end{document}